%% file: arXiv.tex
\documentclass[table,11pt]{article}
\usepackage{amsmath}
\usepackage{graphicx}
\usepackage{psfrag}
\usepackage{epsf}
\usepackage{enumerate}
\usepackage{xcolor}
\usepackage{bm}
\usepackage{soul}

\usepackage[round,longnamesfirst]{natbib}
\usepackage{etoolbox}
\usepackage{mleftright}

\makeatletter
\newcommand\bibstyle@comma{\bibpunct(),a,,}
\newcommand\bibstyle@semicolon{\bibpunct();a,,}
\makeatother

\pretocmd\cite{\citestyle{comma}}\relax\relax
\pretocmd\citep{\citestyle{semicolon}}\relax\relax

\usepackage{url} % not crucial - just used below for the URL
\usepackage{amsthm}
\usepackage{xpatch}
\usepackage{subcaption}
\usepackage{bm}

\usepackage{longtable}
\usepackage{mathrsfs}
\usepackage{tikz}
\usepackage{color} 
\usepackage{amssymb}
\usepackage{latexsym}
\usepackage{xr}
\usepackage{float}
\usepackage{multirow}
\usepackage{dsfont}
\usepackage{booktabs}

\usepackage{mathabx}
\usepackage[ruled]{algorithm2e}
\usepackage{lscape}
\usepackage{rotating}
\numberwithin{equation}{section}

\usepackage{breakcites}
\usepackage{xcolor}
\definecolor{lightblue}{HTML}{044E9E}
\usepackage[breaklinks=true, hyperfootnotes = true,colorlinks,citecolor=lightblue,urlcolor=lightblue, linkcolor=lightblue]{hyperref}

\usepackage{mathtools}
\mathtoolsset{showonlyrefs}

\def\ZZ{{\mathbb Z}}

\def\RR{{\mathbb R}}

\def\PP{\mathbb{P}}
\def\EE{\mathbb{E}}
\def\RR{\mathbb{R}}
\def\NN{\mathbb{N}}

\def\calK{\mathcal K}
\def\clp{\mathcal{P}}

\def\calS{\mathcal S}

\def\TV{\text{TV}}

\def\bbH{\mathbb{H}}

\newcommand{\comment}[1]{}

\newcommand{\Prob}{\mathbb P}

\newcommand{\Cov}[0]{\operatorname{Cov}}
\newcommand{\Var}{\operatorname{Var}} 

\newcommand{\vecop}{\operatorname{vec}}

\newcommand{\diag}{\operatorname{diag}}

\newcommand{\tr}{{\rm{tr}}}

\newcommand{\argmin}{{\rm{argmin}}}

\newcommand{\clx}{\mathcal{X}}

\newcommand{\clb}{\mathcal{B}}

\newcommand{\clk}{\mathcal{K}}
\newcommand{\clf}{\mathcal{F}}
\newcommand{\cln}{\mathcal{N}}

\newcommand{\bmH}{\mathbb{H}}

\usepackage{mleftright}
\usepackage{xparse}

\NewDocumentCommand{\evaluat}{sO{\big}mm}{%
  \IfBooleanTF{#1}
   {\mleft. #3 \mright|_{#4}}
   {#3#2|_{#4}}%
}

\DeclareMathOperator*{\argminlim}{arg\,min}

\usepackage{hhline} % used for the top-most horizontal line! in order to fix the top left corner cell empty
\usepackage{highlight}
\renewcommand{\opacity}{50}

\theoremstyle{plain}
\newtheorem{Thm}{Theorem}[section]
\newtheorem{Lem}[Thm]{Lemma}

\newtheorem{Def}[Thm]{Definition}

\theoremstyle{remark}
\newtheorem{Rem}[Thm]{Remark}

\theoremstyle{definition}

\newtheorem{Gassump}{Assumption}

\newtheorem{Kassump}{Assumption}

\newtheorem{Nassump}{Assumption}

\makeatletter
\xpatchcmd{\proof}{\@addpunct{.}}{\@addpunct{:}}{}{}
\makeatother

\DeclareFontFamily{U}{mathx}{\hyphenchar\font45}
\DeclareFontShape{U}{mathx}{m}{n}{<-> mathx10}{}
\DeclareSymbolFont{mathx}{U}{mathx}{m}{n}
\DeclareMathAccent{\widebar}{0}{mathx}{"73}
\newcommand{\mockalph}[1]{}

\allowdisplaybreaks

\begin{document}

\def\spacingset#1{\renewcommand{\baselinestretch}%
{#1}\small\normalsize} \spacingset{1}

%%%%%%%%%%%%%%%%%%%%%%%%%%%%%%%
\newtheorem*{assumptionBIC*}{\assumptionnumber}
\providecommand{\assumptionnumber}{}
\makeatletter
\newenvironment{assumptionBIC}[2]
 {%
  \renewcommand{\assumptionnumber}{Assumption #1#2}%
  \begin{assumptionBIC*}%
  \protected@edef\@currentlabel{#1#2}%
 }
 {%
  \end{assumptionBIC*}
 }
\makeatother
 %%%%%%%%%%%%%%%%%%%%%%%%%%%%%%%

%%%%%%%%%%%%%%%%%%%%%%%%%%%%%%%%%%%%%%%%%%%%%%%%%%%%%%%%%%%%%%%%%%%%%%%%%%%%%%

\title{Kernel Estimation for Nonlinear Dynamics
\footnote{MD was partially supported by the FAU Emerging Talents Initiative.}}

\author{
Marie-Christine D\"uker \\ FAU Erlangen-N\"urnberg
\and
Adam Waterbury \\ Denison University
}
\date{\today}

\maketitle

\bigskip

\begin{abstract}
\noindent
Many scientific problems involve data exhibiting both temporal and cross-sectional dependencies. While linear dependencies have been extensively studied, the theoretical analysis of regression estimators under nonlinear dependencies remains scarce. This work studies a kernel-based estimation procedure for nonlinear dynamics within the reproducing kernel Hilbert space framework, focusing on nonlinear vector autoregressive models. We derive nonasymptotic probabilistic bounds on the deviation between a regularized kernel estimator and the nonlinear regression function. A key technical contribution is a concentration bound for quadratic forms of stochastic matrices in the presence of dependent data, which is of independent interest. Additionally, we characterize conditions on multivariate kernels that guarantee optimal convergence rates.
\end{abstract}

\input{main}

\small
\bibliographystyle{plainnat}
\bibliography{references}

\end{document}

%% file: main.tex
\section{Introduction}
Many data exhibit nonlinear behaviors that linear models cannot capture. In fields such as finance \citep{benrhmach2020nonlinear}, economics \citep{nyberg2018forecasting}, biology and neuroscience \citep{kato2006statistical,yu2021sparse}, relationships among variables often involve thresholds, feedback loops, or sudden changes that linear models fail to address. These data often display temporal and cross-sectional dependencies, making them well-suited for modeling with nonlinear dynamical systems. However, estimation techniques for time series arising from such systems remain far less developed than those for linear models. In contrast, the statistical learning and machine learning literature provide methodology that allow for the estimation of such nonlinear relationships.

We employ the mathematical framework of reproducing kernel Hilbert spaces (RKHSs) to introduce an estimator for nonlinear dynamical systems with temporal and cross-sectional correlations. Kernel-based techniques allow to make inferences in a high-dimensional feature space mapped implicitly by a kernel function. This technique is integral to many statistical learning algorithms, including the support vector machine (SVM) and the kernel principle component analysis (KPCA) algorithms (see \cite{hastie2009elements}),  but has not been explored in the context of nonlinear time series. Our estimation procedure and theoretical results are phrased for nonlinear vector autoregression models but also cover stochastic regression models. We treat the stochastic regression model as a special case of nonlinear vector autoregression; see Section \ref{se:stochastic regression} for further discussion.

\textit{Nonlinear vector autoregression:} Suppose $\{X_t\}_{t \in \ZZ}$ follows a causal nonlinear VAR model of order $p$. To be more precise, we consider an $\RR^d$-valued sequence $\{X_{t}\}_{t \in \ZZ}$, where $X_{t} = (X_{1,t}, \dots, X_{d,t})'$ is  defined according to  
\begin{align} \label{eq:nonlinearVAR_2}
X_t &= g(X_{t-1}, \dots, X_{t-p}) + \varepsilon_{t}= \left(g_{1}(X_{t-1}, \dots, X_{t-p}), \dots, g_{d}(X_{t-1}, \dots, X_{t-p})\right)' + \varepsilon_t,
\end{align}
for a function $g: \RR^{dp} \to \RR^d$ and a suitable suitable noise sequence $\{\varepsilon_t\}_{t\in\ZZ}$. The model \eqref{eq:nonlinearVAR_2} is quite general and allows for nonlinear dependencies across multiple lags ($p$) and across space ($d$). The goal is to estimate the function $g$, which allows for downstream tasks such as forecasting,  interpretation of dynamics, and clustering.
Note that the model \eqref{eq:nonlinearVAR_2} also covers additive dynamics of the form 
\begin{align}\label{eq:nonlinearVAR}
X_t = \sum\nolimits_{j=1}^{p} h_{j}(X_{t-j}) + \varepsilon_{t},
\end{align}
where $h_1,\dots,h_d : \RR^{d} \to \RR^d$ are nonlinear regression functions.
To estimate the function $g$ in \eqref{eq:nonlinearVAR_2}, we employ a  regularized least squares estimator also known as Kernel Ridge Regression (KRR).
KRR is a popular technique in supervised learning and has been used to avoid overfitting in regression problems. However, to the best of our knowledge, there is no literature on the use of KRR estimators for temporal models. In particular, we are the first to provide statistical guarantees under temporal and cross-sectional dependencies for such an estimator. On the way, we formalize the use of multivariate RKHSs and corresponding Mercer representations when kernels are evaluated at multivariate data. In contrast to regression problems, the use of temporal models such as nonlinear vector autoregression makes it desirable to allow for unbounded errors, which requires kernels defined on unbounded spaces. 

The theoretical difficulties in this work are the development of concentration results for random quadratic forms of a multivariate empirical kernel matrix. We establish nonasymptotic high probability bounds for KRR estimators by proving concentration results for quadratic forms of dependent data. To be more precise, our results establish a Hoeffding-type inequality for quadratic forms involving empirical kernel matrices, and are of independent interest.
We provide conditions under which the estimator achieves the optimal convergence rate in the sense that the results align with the rate for quadratic forms of deterministic matrices; see  \cite{rudelson2013hanson}.

\textit{Literature review:} 
The literature on nonlinear regression models and corresponding statistical guarantees is sparse. The existing literature has considered regression settings \citep{liu2023estimation} and additive models incorporating sparsity assumptions; see \cite{yuan2016minimax,zhou2018non}, who consider models with additive dynamics similar to \eqref{eq:nonlinearVAR}. 
To the best of our knowledge there is no work providing statistical guarantees for regularized estimators of  nonlinear vector autoregression models.

The works \cite{smale2005shannon} and \cite{liu2023estimation} are the closest to our paper and served as inspiration for our estimation procedure. \cite{smale2005shannon} introduce a KRR estimator for a nonlinear regression model and derive conditional concentration results, effectively looking at nonrandom predictors. \cite{liu2023estimation} provide estimators for higher order derivatives of the regression function based on a KRR estimator. Our main result can recover their results but also allows for much more general settings. In particular, our results allow for nonlinear temporal and cross-sectional dependencies. Furthermore, our class of feasible kernels includes kernel functions on noncompact spaces and allows for multivariate inputs.

Regularization techniques like ridge- and lasso-type penalties have been studied extensively for linear regression problems. Early works study linear regression models with deterministic predictors; see \cite{loh2012high} who provide consistency results for lasso estimators.
Later generalizations cover stochastic regression and linear vector autoregression models; see \cite{basu2015regularized} for lasso and \cite{ballarin2024ridge} for ridge estimation.

Our theoretical results crucially rely on concentration results for functionals of Markov chains. Recent literature studies nonlinear functionals of Markov chains under a range of boundedness and smoothness conditions; see \cite{adamczak2015exponential,adamczak2015concentration,paulin2015concentration,chen2017concentration,alquier2019exponential,fan2021hoeffding}.

Given the kernel-based estimation approach, our work naturally connects to the literature on U-statistics, particularly concentration inequalities for U-statistics. 
The problem is well studied for kernels evaluated at i.i.d. samples; see 
\cite{arcones1993limit,arcones1995bernstein}. A survey can be found in \cite{pitcan2017note}, and recent works attempt to generalize those results in several directions. For instance, \cite{chakrabortty2018tail} consider unbounded kernels and \cite{duchemin2023concentration} 
study U-statistics for Markov chains  and characterize assumptions on the chain that ensure optimal convergence rates. \cite{borisov2015note} and \cite{han2018exponential} established  exponential inequalities for U-statistics of order two and larger in time series under mixing conditions. These findings were later refined by \cite{shen2020exponential}, which introduced a Hanson-Wright-type inequality for both V- and U-statistics under conditions on the time-dependent process that are more practical to verify.

\textit{Outline of the paper:} The rest of the paper is organized as follows. In Section \ref{se:prelim}, we introduce a kernel-based  regularized estimator and state our main  assumptions. Section \ref{se:concentration} formalizes our main results and includes a discussion of these results and associated assumptions. The section is supplemented with several examples that illustrate that our assumptions are satisfied by a wide range of models.
In Section \ref{se:auxiliary results}, we state some more technical tools and auxiliary results.
Finally Sections \ref{se:proofs} and \ref{se:ProofsGaussianKernel} provide the proofs of our results.

\textit{Notation:} We write $\NN  \doteq \{1,2,\dots\}$ and $\NN_0 \doteq \NN \cup \{0\}$. For measurable spaces $(\clx_1, \clf_1)$ and $(\clx_2, \clf_2)$, we let $\clb(\clx_1:  \clx_2)$ denote the space of measurable functions from $\clx_1$ to $\clx_2$. {Set $\clx \subseteq \RR^{d}$} and for a kernel {$K : \clx^p \times \clx^p \to \RR$, $x, y \in \clx^p$,} we use the notation $K(x, y)$ and $K(\vecop(x), \vecop(y))$ interchangeably, where $\vecop(\cdot)$ stacks each column of a matrix. For a separable Hilbert space $\bbH \subseteq \clb( \clx : \RR)$  with inner product $\langle \cdot, \cdot \rangle_{\bbH}$ and orthonormal basis $\{ \phi_k, \; k \in \NN\}$, given an $\bbH$-valued random variable $X$, we define $\EE_{\bbH}(X) \in \bbH$ as $\EE_{\bbH}(X) \doteq  \sum\nolimits_{k=1}^{\infty} \EE\left( \langle  X, \phi_k\rangle_{\bbH}\right) \phi_k$, 
whenever the series is convergent in $\bbH$. If $x \in \RR^m$ for some $m \in \NN$, we let $\| x\| = (\sum\nolimits_{i=1}^{m} x_i^2)^{1/2}$. For $M \in \RR^{m \times n}$, we let $\| M\| = ( \sum\nolimits_{i=1}^m \sum\nolimits_{j=1}^{n} M_{ij}^2)^{1/2}$ denote the Frobenius norm of $M$.  Other norms are denoted using subscripts. In particular, for a function $f : \RR^m \to \RR^n$, we set $\|f\|_{\infty} = \sup\nolimits_{x \in \RR^m}\|f(x)\|_{\infty}$, where for a vector $y \in \RR^n$, we let $\|y\|_{\infty} = \sup\nolimits_{i=1,\dots,n} |y_i|$.  For a Polish space $\mathcal{S}$, we write $\clp(\mathcal{S})$ to denote the space of probability measures on $\mathcal{S}$, and for two finite measures $\mu, \nu$ on $\mathcal{S}$, we write $(\mu - \nu)$ to denote the signed measure defined as $(\mu - \nu)(A) = \mu(A) - \nu(A)$. Similarly, we write $|\mu-  \nu|$ to denote the finite measure defined as $|\mu - \nu|(A) = |\mu(A) - \nu(A)|$. The total variation distance between $\mu$ and $\nu$ is given by $\| \mu - \nu\|_{\TV} = \sup\nolimits_{A\in\clb(\calS)}|\mu(A)-\nu(A)|$. For $\pi \in \clp(\RR^d)$, ${\PP}_{\pi}(X_1 \in A_1) = \pi(A_1)$ is the law under which $X_1 \sim \pi$. For $x \in \clx$, we write $\PP_x$ to denote $\PP_{\delta_x}$, where $\delta_x$ is the Dirac measure at $x$. 
For $n \in \NN$ and $k_1,\dots,k_{j} \in \NN_0$ such that $\sum\nolimits_{i=1}^{j} k_i = n$ and  $\bm{k} = (k_1,\dots,k_{j})$, we write $\binom{n}{\bm{k}} = \frac{n}{\bm{k}!}$, where $\bm{k}! = k_1! \cdots k_{j}!$. For $k \in \NN_0$, let $\bm{k} + k = (k_1 + k, \dots, k_j + k)$. We use $\Gamma$ to denote the gamma function and $B$ to denote the beta function.

\section{Preliminaries}\label{se:prelim}
In this section, we establish notation and present preliminary results to introduce a regularized estimator within the RKHS framework. Recalling the model \eqref{eq:nonlinearVAR_2}, it is convenient to introduce the corresponding $\clx^p$-valued nonlinear VAR model of order one and treat it as a Markov chain. This $\clx^p$-valued process $\{Y_t\}$ is defined by 
\begin{equation}\label{eq:defyn1b}
Y_t \doteq   
    \begin{pmatrix}
        X'_{t}, 
        \dots, 
        X'_{t-p+1}
    \end{pmatrix}', \quad t = p,\dots,T, 
\end{equation}
and, for $t = p, \dots, T$, we define  
\begin{equation}\label{eq:Ytdef}
\begin{pmatrix}
        X_{t}\\
        \vdots\\
        X_{t-p+1}
    \end{pmatrix}
    =   \begin{pmatrix}
        g(X_{t-1},\dots,X_{t-p})\\
        X_{t-1}\\
        \vdots\\
        X_{t-p+1}
    \end{pmatrix} + \begin{pmatrix}
        \varepsilon_{t} \\
        0 \\
        \vdots \\
        0
    \end{pmatrix}
    \hspace{0.2cm}
    \text{ or }
    \hspace{0.2cm}
    Y_{t} = G_Y(Y_{t-1}) + \xi_t.
\end{equation}
For notational convenience, we sometimes write 
\begin{equation}\label{eq:defyt2b}
Y_t = (Y_{t,1}',\dots, Y_{t,p}')', \quad Y_{t,j} = X_{t +1 - j}, \; j = 1,\dots,p,
\end{equation}
so that, in particular, $Y_{t,1} = X_t$. In Lemma \ref{prop:geomergodic} it is shown that, under some mild assumptions, $\{Y_t\}$ is a $\clx^p$-valued geometrically ergodic Markov chain and therefore has a unique stationary distribution $\pi \in \clp(\clx^p)$. We let
\[
L^2(\clx^p, \pi ) = \left\{ f \in \mathcal{B}( \clx^p : \RR ) : \int_{\clx^p} (f(y))^2 \pi(dy) < \infty \right\}
\]
denote the associated space of square-integrable functions.
 An RKHS is defined through a kernel, namely a symmetric and positive definite function $K : \clx^p \times  \clx^p \to \RR$. More precisely, $K$ is a kernel if  
\[
K(x,y) = K(y,x), \quad x,y \in \clx^p,
\]
and $\sum_{i,j=1}^d c_i c_j K(z_i, z_j) \ge 0$ for all $z_1,\dots, z_d \in \clx^p$ and  $c_1,\dots, c_d \in \RR$, with equality if and only if $c_i = 0$ for all $i = 1,\dots,d$. For kernels $K_1,\dots,K_d : \clx^p \times \clx^p \to \RR$, we let $\bbH_i \subseteq L^2(\clx^p, \pi )$ denote the RKHS associated with $K_i$. That is $(\bbH_i, \| \cdot \|_{\bbH_i})$ is the Hilbert space defined as,
\[
\bbH_i = \{ f \in L^2(\clx^p,\pi):  \text{ for each } x \in \clx^p, \; f(x) = \langle f , K_i(x,\cdot) \rangle_{\bbH_i} \},
\]
where $\|\cdot\|_{\bbH_i}$ is the RKHS inner product associated with $K_i$ (see Section \ref{se:repHil}).
For the multivariate kernel 
\begin{equation} \label{eq:diagkernel1}
    \clk(x,y) = \text{diag}(K_1(x,y),\dots,K_d(x,y)), \quad x,y \in \clx^p \times \clx^p
\end{equation}
we denote the associated RKHS by $(\bbH, \| \cdot \|_{\bbH})$ and note that $\bbH$ can be  identified with a subset of $L^{2,d}(\clx^p, \pi ) \doteq  \bigtimes_{i=1}^d L^2(\clx^p, \pi)$. Further discussion of RKHS inner products and multivariate RKHS can be found in Section \ref{se:repHil}.

\subsection{Kernel ridge regression}\label{sec:krrnonlinearvar}
In this section, we introduce a regularized kernel-based estimator for the regression function $g$ in \eqref{eq:nonlinearVAR_2}. To estimate this function, we employ a ridge-type objective function of the form 
\begin{equation} \label{eq:objectivefunction1}
\begin{split}
    \widehat{g}_T &= \argmin_{g \in \bbH} \left\{ \frac{1}{T} \sum_{t=p+1}^{T} (X_t - g(X_{(t-p):(t-1)}))'(X_t - g(X_{(t-p):(t-1)}))  + \lambda \| g \|^2_{\bbH} \right\}\\
    &=  \argmin_{g \in \bbH} \left\{ \frac{1}{T} \sum_{t=p+1}^{T} (Y_{t,1} - g(Y_{t-1}))'(Y_{t,1} - g(Y_{t-1}))  + \lambda \| g \|^2_{\bbH} \right\},
    \end{split}
\end{equation}
where $X_{(t-p):(t-1)} \doteq (X_{t-1},\dots, X_{t-p})$ and 
$(\bbH, \| \cdot \|_{\bbH})$ is the RKHS induced by $\mathcal{K}$ (see Section \ref{se:repHil}). The second identity in \eqref{eq:objectivefunction1} uses the notation introduced in \eqref{eq:defyn1b}.

Note that the model \eqref{eq:nonlinearVAR_2} can be written more compactly as
\begin{equation} \label{eq:model_matrix_form}
  \mathcal{Y}_X \doteq  \begin{pmatrix}
        X'_{p+1} \\
        \vdots \\
        X'_{T}
    \end{pmatrix}
    = 
    \begin{pmatrix}
        g(X_p,\dots, X_1)' \\
        \vdots \\
        g(X_{T-1}, \dots, X_{T-p})'
    \end{pmatrix}
    +
    \begin{pmatrix}
        \varepsilon'_{p+1} \\
        \vdots \\
        \varepsilon'_{T}
    \end{pmatrix} \doteq  G_X(g) + \mathcal{E}.
\end{equation}
Furthermore, \eqref{eq:model_matrix_form} can be written in vectorized form as 
\begin{equation}\label{eq:vectorizedregressionproblem1}
   Y \doteq  \vecop(\mathcal{Y}_{X}) = \vecop(G_{X}(g)) + \vecop(\mathcal{E})  \doteq G_{g}(X) + \eta,
\end{equation}
where
\begin{equation}\label{eq:xdef635}
X  \doteq (X_1,\dots,X_T) \in \RR^{d \times T}
\hspace{0.2cm}
\text{ and }
\hspace{0.2cm}
\eta \doteq \vecop(\mathcal{E}).
\end{equation}
The goal is to rewrite the objective function \eqref{eq:objectivefunction1} using a representer theorem in order to obtain an explicit representation for $\hat{g}_T$ akin to the standard ridge regression estimator in linear regression problems.  For the kernels $K_1, \dots, K_d$, there are, by the representer theorem \citep{scholkopf2001generalized}, 
\[
\widehat{\alpha}_i = (  \widehat{\alpha}_{i,p+1}, \dots, \widehat{\alpha}_{i,T})' \in \RR^{(T-p)\times 1},
\hspace{0.2cm}
i=1,\dots, d,
\]
such that, for $z \in \clx^p$, 
\begin{align} \label{eq:g_i}
\widehat{g}_{i,T}(z) & = 
\sum\nolimits_{t=p+1}^T \widehat{\alpha}_{i,t} K_i(z, X_{(t-p ): (t-1)})
= K_{X,i}(z)' \widehat{\alpha}_i,
\end{align}
where, for $z  \in \clx$, $K_{X,i}(z) \in \RR^{ (T-p) \times 1}$ is defined by 
\[
K_{X, i}(z) = \begin{pmatrix}
    K_i(z, X_{1:p})\\
    \vdots\\
    K_i(z,X_{(T-p):(T-1)})
\end{pmatrix}.
\]
In particular, using \eqref{eq:g_i}, the minimizer in \eqref{eq:objectivefunction1} can be written explicitly as
\begin{align} \label{eq:alphaandKvector}
    \widehat{g}_{T}(z) 
    &= ( \widehat{g}_{1,T}(z), \dots, \widehat{g}_{d,T}(z))'
    = 
    \left(
    K_{X,1}(z)' \widehat{\alpha}_{1}
,\dots,
     K_{X,d}(z)' \widehat{\alpha}_{d}
\right)'
 = {\mathcal{K}}_{X}(z) \widehat{\alpha},
\end{align}
where, for each $z \in \clx^p$, $\clk_X(z) \in \RR^{d \times d(T-p)}$ is given by 
\begin{equation}\label{eq:empkwithargument}
{\mathcal{K}}_{X}(z) \doteq \diag(K_{X, 1}(z)', \dots, K_{X, d}(z)'),
\end{equation}
and
$
\widehat{\alpha} \doteq \vecop(\widehat{\alpha}_1, \dots, \widehat{\alpha}_d) \in \RR^{d(T-p)\times 1}$.
Recall \eqref{eq:model_matrix_form} and note that, for $i=p+1,\dots, T$, using \eqref{eq:alphaandKvector}, the estimator for the $(i-p)$-th row of $G_X$ in \eqref{eq:model_matrix_form} can be written as
\begin{equation} \label{eq:ghat}
\begin{aligned}
\widehat{g}_{T}(X_{i-1}, \ldots, X_{i-p}) 
&= 
\Bigg(
\sum\nolimits_{t=p+1}^T \widehat{\alpha}_{1,t} K_1(X_{(i-p):(i-1)}, X_{(t-p ): (t-1)})
,\dots,
\\&\hspace{4cm}
\sum\nolimits_{t=p+1}^T \widehat{\alpha}_{d,t} K_d(X_{(i-p):(i-1)}, X_{(t-p ): (t-1)})
\Bigg)'.
\end{aligned}
\end{equation}
To rewrite  the data matrix $G_X(g)$ in \eqref{eq:model_matrix_form} as in \eqref{eq:ghat}, 
consider the empirical kernel matrices ${K}_1(X,X ), \dots,{K}_d(X, X) \in \RR^{(T-p) \times (T-p)}$,
with  ${K}_i(X,X ) \doteq [({K}_{i}(X,X))_{s,t}]_{s,t = p+1,\dots,T}$  defined as 
\begin{equation*}\label{eq:kernelmatrixdef1}
({K}_{i}(X,X))_{s,t}  = K_i(X_{(s-p):(s-1)}, X_{(t-p):(t-1)}), \quad s,t = p+1,\dots,T.
\end{equation*}
We also define an empirical kernel matrix
\begin{equation}\label{eq:diagempkernelmatrix}
{\mathcal{K}}(X,X) = \diag({K}_1(X,X), \dots,{K}_d(X,X) ) \in \RR^{d(T-p) \times d(T-p)} .
\end{equation}
Then, vectorizing the estimated counterpart of $G_X(g)$ in \eqref{eq:model_matrix_form}, we have
\begin{align}
 \vecop
    \begin{pmatrix}
        \widehat{g}_T(X_p,\dots, X_1)' \\
        \vdots \\
        \widehat{g}_T(X_{T-1}, \dots, X_{T-p})'
    \end{pmatrix}
    &=
\vecop
\left(
K_1 (X,X) \widehat{\alpha}_1
,\dots,
K_d(X,X) \widehat{\alpha}_d
\right) \nonumber\\
& = 
\diag(K_1(X,X), \dots,K_d(X,X) )
\vecop
\begin{pmatrix}
(\widehat{\alpha}_1, \hdots, \widehat{\alpha}_d)
\end{pmatrix} \nonumber \\
&= \mathcal{K}(X,X) \widehat{\alpha}.
\label{eq:kerneltimesalphahat}
\end{align}
Using \eqref{eq:ghat} and \eqref{eq:kerneltimesalphahat}, we can reduce \eqref{eq:objectivefunction1} to an equivalent convex objective function in $\RR^{d p \times (T-p)}$, of the form
\begin{align} 
    \argminlim\nolimits_{ \alpha \in \RR^{d(T-p)} } \left\{ \frac{1}{T} \| Y - \mathcal{K}(X,X) \alpha  \|^2_2 + \lambda {\alpha}' \mathcal{K}(X,X) {\alpha}  \right\} \label{eq:objectivefunction3}. 
\end{align}
Finally, the optimization problem in \eqref{eq:objectivefunction3} is solved by the estimator
\begin{equation*}
    \widehat{\alpha} = 
(\mathcal{K}(X,X) \mathcal{K}(X,X) + 
\lambda T \mathcal{K}(X,X))^{-1} \mathcal{K}(X,X) \vecop(X)
\end{equation*}
so it follows from \eqref{eq:alphaandKvector} that the minimizer in \eqref{eq:objectivefunction1} is given by, for $z \in \clx^p$, 
\begin{equation}\label{eq:ghatkernelrep1}
\begin{aligned}
    \widehat{g}_T(z) = 
    \mathcal{K}_{X}(z)  \widehat{\alpha} &=
    \mathcal{K}_{X}(z)  \left(\mathcal{K}(X,X) \mathcal{K}(X,X) + \lambda T \mathcal{K}(X,X)\right)^{-1} \mathcal{K}(X,X) Y \\&=
    \mathcal{K}_{X}(z)  (\mathcal{K}(X,X) + \lambda T I_{d(T-p)})^{-1} Y.
\end{aligned}
\end{equation}

On the population level, an objective function analogous to \eqref{eq:objectivefunction1} can be written as 
\begin{equation}\label{eq:glambdaoptimize}
g_{\lambda} = \argminlim\nolimits_{\widetilde{g} \in \bbH} \left\{ \| \widetilde{g} - g\|_2^2 + \lambda \|\widetilde{g}\|_{\bbH}^2\right\}.
\end{equation}
Recalling the definition of $L^{2,d}(\clx^p,\pi)$ from Section \ref{se:prelim}, standard results (see, e.g., \cite{smale2005shannon,liu2023estimation}) show that an explicit representation of $g_{\lambda}$ is given by 
\begin{equation}\label{eq:glamdef1}
g_{\lambda}(z) = ( L_{\mathcal{K}} + \lambda I )^{-1} L_{\mathcal{K}} g(z), \quad z \in \clx^p,
\end{equation}
where the integral operator $L_{\mathcal{K}}$ on $L^{2,d}(\clx^p,\pi)$ is defined as
\begin{equation}\label{eq:lkiintegralop2}
L_{\mathcal{K}} (f) (z) = \int_{ \clx^p} \mathcal{K}(z, y) f (y)  \pi(dy), \quad z \in \clx^p, \; f = (f_1,\dots,f_d)' \in L^{2,d}(\clx^p,\pi),
\end{equation}
and where $I$ is the identity operator in $L^{2,d}(\clx^p,\pi)$.

\subsection{Assumptions} \label{se:ass}
In this section we collect the assumptions needed to state our main results. The assumptions can be separated into three sets: the first concerns regression function $g$, the second pertains to the kernel used for estimation, and the third addresses the noise sequence.

\begin{Gassump}[Bounded dynamics]\label{ass:gbounded}
There is some $M_g \in (0,\infty)$ such that $\| g \|_{\infty} \le M_g$.
\end{Gassump}
Assumption \ref{ass:gbounded} helps ensure that the model in \eqref{eq:nonlinearVAR_2} is stationary and geometrically ergodic; see Section \ref{se:geom-erg}.
The next set of assumptions is stated for the kernels $K_1,\dots,K_d$.

\begin{Kassump}[Bounded kernel]\label{ass:kernelbound}
For each $j = 1,\dots,d$ there  is some $\kappa_j \in (0,\infty)$ such that 
\[
\sup\nolimits_{x \in \clx^p} | K_j( x, x) | \le \kappa_j^2.
\]
Furthermore, we set $\kappa \doteq (\sum\nolimits_{j=1}^d \kappa_j^2)^{1/2}$.
\end{Kassump}

\begin{Kassump}[Mercer expansion] \label{ass:separable_kernel_multi}
\;
\label{ass:separable_kernel_multi_parta}     The kernels $K_1,\dots,K_d$ in \eqref{eq:diagkernel1} are separable. That is, for some $M \in \NN \cup \{\infty\}$,  there is a collection of eigenvalues $\lambda_{i,k}$ and eigenfunctions $\phi_{i,j,k}$, $i=1,\dots, d$, $k=1,\dots,M$, $j=1,\dots,N(k)$, such that $\|\phi_{i,j,k}\|_{\infty} < \infty$, and
\begin{equation}\label{eq:mercerexpansion}
    K_i(x,y) =    \sum\nolimits_{k=1}^{M} \lambda_{i,k}  \sum\nolimits_{j=1}^{N(k)} \phi_{i,j,k}(x)  \phi_{i,j,k}(y),
    \end{equation}
where 
\begin{equation} \label{eq:def-N(k)}
N(k) =  | \mathcal{N}(k) |
= \binom{k+dp-1}{dp-1},
\end{equation}
with
\begin{equation} \label{eq:def-mathcalN(k)}
\mathcal{N}(k) = \left\{ (n_{1,1},\dots,n_{1,p}, \dots,n_{d,1},\dots, n_{d,p})  \in \NN_0^{dp}  : \sum\nolimits_{i=1}^{d} \sum\nolimits_{r=1}^{p} n_{i,r} = k\right\}.
\end{equation}
\end{Kassump}

\begin{Rem}
Assumption \ref{ass:separable_kernel_multi} gives a general Mercer representation that allows the kernel to have multivariate input variables. When $d = p = 1$,  \eqref{eq:mercerexpansion} reduces to the standard univariate Mercer kernel expansion, namely  $
K_1(x,y) = \sum\nolimits_{k=1}^{M} \lambda_{k} \phi_{k}(x)\phi_k(y)$.
\end{Rem}

\begin{Kassump}[Eigenfunction growth] \label{ass:eigenfunctiongrowth}
There are $b_1 \in (0,\infty), b_2 \in \RR$ such that for each $M \in \NN$ and   $i = 1,\dots,d$, with 
\begin{equation}\label{eq:assumptionbetasdef}
\beta_{i,j,k} \doteq \lambda^{1/2}_{i,k} \|\phi_{i,j,k}\|_{\infty},
\end{equation}
we have  
\begin{equation}\label{eq:assumptionbetasdef-2}
\sum\nolimits_{k=1}^M \sum\nolimits_{j=1}^{N(k)} \beta^2_{i,j,k} \le b_1 M^{b_2}.
\end{equation}
\end{Kassump}

\begin{Kassump}[Kernel moment bound]\label{ass:momentboundtail}
There is an absolutely summable sequence $\{\alpha_k\}$ such that for all $k \in \NN$, $t \ge p +1$, and $i = 1,\dots,d$, 
\begin{equation}\label{eq:momentboundtaileqn1}
   \lambda_{i,k} \sum\nolimits_{j=1}^{N(k)}  \EE((\phi_{i,j,k}(Y_{t-1}))^2)  = \lambda_{i,k} \sum\nolimits_{j=1}^{N(k)}  \EE((\phi_{i,j,k}(X_{(t-p):(t-1)}))^2) \le \alpha_k.
\end{equation}
Additionally, there are $\beta_1,\beta_2 \in (0,\infty)$ and  an increasing  sequence $\{M(T)\}_{T \in \NN}$ such that for each $T \in \NN$,
\begin{equation}\label{eq:momentboundtaileqn2}
\sum\nolimits_{k=M(T)}^{\infty} \alpha_k \leq \beta_1 T^{-\beta_2}.
\end{equation}
\end{Kassump}

\begin{Rem}
   In Assumption \ref{ass:separable_kernel_multi} we only assume that each eigenfunction is bounded. If $M =\infty$, then there may not be a uniform bound over all eigenfunctions. This is addressed through Assumptions \ref{ass:eigenfunctiongrowth} and \ref{ass:momentboundtail}. Assumption \ref{ass:eigenfunctiongrowth} is trivially satisfied when the kernel has a finite Mercer expansion.
\end{Rem}

\begin{Nassump}[Sub-Gaussian noise] \label{ass:subgaussiannoise}
    For each $t \in \NN$, $i=1,\dots, d$, $\varepsilon_{i,t}$ is sub-Gaussian with variance proxy $\sigma^2_i \in (0,\infty)$. Additionally, $\varepsilon_{i,t}$ admits a density $\psi$ such that for each compact set $A \subseteq \text{supp}(\varepsilon_{i,t})$,  $\inf\nolimits_{x\in A}\psi(x) > 0$.
\end{Nassump}

We refer to Proposition 2.5.2 in \cite{vershynin2018high} for different characterizations of sub-Gaussianity. We also state two slightly more restrictive assumptions for the discussion of our convergence rates below. Both are special cases of Assumption \ref{ass:subgaussiannoise}.

\begin{Nassump}[Gaussian noise] \label{ass:gaussiannoise}
    For each $t \in \NN$, we have that $\varepsilon_t \sim \cln( 0, \Sigma)$, where $\Sigma = \diag(\sigma_1^2,\dots,\sigma^2_d)$.
\end{Nassump}
\begin{Nassump}[Bounded noise] \label{ass:boundednoise}
    Assumption \ref{ass:subgaussiannoise} is satisfied, and there is some $L_{\varepsilon} > 0$ such that for each $t \in \NN$, $\| \varepsilon_t\|_{\infty} \le L_{\varepsilon}$.
\end{Nassump}

\begin{Rem}
    Under Assumptions \ref{ass:gbounded} and \ref{ass:boundednoise}, we can, without loss of generality, assume that $\clx$ is a bounded subset of $\RR^d$. 
\end{Rem}

\section{Concentration results for kernel functions} \label{se:concentration}
In Section \ref{se:main-results} we present concentration results for kernel regression estimators under temporal dependence. Then, in Section \ref{se:examples}, we present several applications in which our  assumptions are satisfied.

\subsection{Main results} \label{se:main-results}
Our main result (Theorem \ref{th:main_consistency}) provides a high probability bound that controls the deviation of the ridge estimator $\widehat{g}_T$ in \eqref{eq:ghatkernelrep1} from the true nonlinear regression function $g$. Recall the function $g_{\lambda}$ introduced in \eqref{eq:glamdef1}.

\begin{Thm} \label{th:main_consistency}
Suppose Assumptions \ref{ass:gbounded}, \ref{ass:kernelbound}--\ref{ass:momentboundtail}, and \ref{ass:subgaussiannoise}.
There are constants $c_1,c_2 \in (0,\infty)$ and $L, \operatorname{c_0}  \in (0,\infty)$ such that if 
    \begin{equation} \label{eq:delta-Th-4.1}
        \delta \geq \frac{1}{\lambda} \sqrt{\frac{\log(T)}{T}} L^{\frac{b_2}{2}} \gamma \operatorname{C_{0}}(g)
        + \| g_{\lambda} - g\|_{\infty},
        \hspace{0.2cm}
        \operatorname{C_{0}}(g) = \operatorname{c_0} \max\{\kappa^2 \| g \|_{\infty},1\}
    \end{equation}
    and
    \begin{equation}\label{eq:gamma-Th-4.1}
        \gamma \geq \sqrt{4 \sigma^2 \log(dT)},
    \end{equation}
then, with probability at least $1- c_1 T^{-c_2}$,
    \begin{equation*}
    \|\widehat{g}_T - g \|_\infty \leq \delta.
\end{equation*}
If Assumption \ref{ass:separable_kernel_multi} is satisfied with $M = \infty$, then we take $L \doteq M(T)$, where $M(T)$ is as in Assumption \ref{ass:momentboundtail}. If Assumption \ref{ass:separable_kernel_multi} is satisfied for some $M < \infty$, then we take $L \doteq M$.
\end{Thm}

Note that Theorem \ref{th:main_consistency} is expressed in terms of $\operatorname{C_{0}}(g)$ and a constant $\operatorname{c_0}$.  
Here $\operatorname{C_{0}}(g)$ depends on $g$, while $\operatorname{c_0}$ subsumes all constants including $d$ and $p$ and may change throughout the proofs.

In view of Theorem \ref{th:main_consistency}, it only remains to control the quantity $\| g_{\lambda} - g\|_{\infty}$ in \eqref{eq:delta-Th-4.1}. While this deterministic term is, for general functions $g$, difficult to control, an estimate is provided in Theorem \ref{thm:dev.deterministic.bound} that, under the somewhat restrictive Assumption \ref{ass:boundednoise}, bounds this term. 

Theorems \ref{th:finiteMercer} and \ref{th:concentration_quafratic_form} below provide concentration inequalities for certain quadratic forms arising in the proof of Theorem \ref{th:main_consistency} and may be of independent interest. To state these results, we introduce functions $c_{1,T}$ and $c_{2,T}$ given by  
\begin{equation} \label{eq:functions-g1-g2}
    c_{1,T}(\delta, \gamma, M) = 
    \frac{T}{\gamma^2} \left( \delta (d b_1 M^{b_2} )^{-\frac{1}{2}}   - \frac{\sigma }{\sqrt{T}} \right)^2,
    \hspace{0.2cm}
    c_{2,T}(\gamma) = \left( \gamma - \sqrt{2 \sigma^2 \log(dT)} \right)^2/(2\sigma^2).
\end{equation}
Note that Theorems \ref{th:finiteMercer} and \ref{th:concentration_quafratic_form} address the cases when Assumption \ref{ass:separable_kernel_multi} is satisfied with finite {$(M < \infty)$} and infinite {$(M = \infty)$} Mercer representations, respectively. Kernels with an infinite Mercer representation require the additional Assumption \ref{ass:momentboundtail} on the variance of the eigenfunctions evaluated at the data.  

\begin{Thm} \label{th:finiteMercer}
Suppose Assumptions \ref{ass:gbounded}, \ref{ass:kernelbound}--\ref{ass:eigenfunctiongrowth}, and \ref{ass:subgaussiannoise}.
   Then, for any $\delta, \gamma  \in (0,\infty)$, there is a $c \in (0,\infty)$, such that for all $T$, 
    \begin{equation*}
        \Prob \left( \eta' \mathcal{K}(X,X) \eta > \delta^2 \right)
        \leq
        d\exp\left(- c c_{1,T}(\delta, \gamma, M) \right)
        +
        d\exp\left(- c_{2,T}(\gamma) \right).
    \end{equation*}
There are constants $c_1,c_2 \in (0,\infty)$ and $M, \operatorname{c_0} \in (0,\infty)$ such that if 
    \begin{equation} \label{eq:delta-Th-4.2}
        \delta \geq \operatorname{c_0} \sqrt{\frac{\log(T)}{T}} M^{\frac{b_2}{2}} \gamma,
        \end{equation}
        and
        \begin{equation}\label{eq:gamma-Th-4.2}
        \gamma \geq \sqrt{4 \sigma^2 \log(dT)},
    \end{equation}
then, with probability at least $1- c_1 T^{-c_2}$, 
\begin{equation}\eta' \mathcal{K}(X,X) \eta \leq \delta^2.
\end{equation}
\end{Thm}

While Theorem \ref{th:finiteMercer} provides a concentration inequality for a quadratic form of an empirical  kernel matrix whose associated kernel has a finite Mercer representation, the next theorem assumes an infinite representation.

\begin{Thm} \label{th:concentration_quafratic_form}
Suppose Assumptions \ref{ass:gbounded},  \ref{ass:kernelbound},  \ref{ass:separable_kernel_multi} with $M=\infty$, \ref{ass:eigenfunctiongrowth}, \ref{ass:momentboundtail}, and \ref{ass:subgaussiannoise}. Let $M(T)$ be as in Assumption \ref{ass:momentboundtail}. Then, for any $\delta, \gamma \in (0,\infty)$, there is a $c \in (0,\infty)$ such that for all $T$, with $L \doteq M(T)$,
\begin{equation} \label{eq:delta-Th-4.3-1}
    \Prob \left( \eta' \mathcal{K}(X,X) \eta > \delta^2 \right)
    \leq
    d\exp\left(- c c_{1,T}(\delta, \gamma, L) \right)
    +
    d\exp\left(- c_{2,T}(\gamma) \right)
    +
    \frac{ 2 d^2 \sigma^2}{ \delta^2 T} \beta_1 T^{-\beta_2}.
    \end{equation}    
There are constants $c_1,c_2 \in (0,\infty)$ and $\operatorname{c_0} \in (0,\infty)$ such that if 
    \begin{equation} \label{eq:delta-Th-4.3}
        \delta \geq \operatorname{c_0} \sqrt{\frac{\log(T)}{T}} L^{\frac{b_2}{2}} \gamma,
        \end{equation}
        and
    \begin{equation}\label{eq:gamma-Th-4.3}
        \gamma \geq \sqrt{4 \sigma^2 \log(dT)}
    \end{equation}
then, with probability at least $1- c_1 T^{-c_2}$, 
\begin{equation} \label{eq:eta-quadratic-form}
\eta' \mathcal{K}(X,X) \eta \leq \delta^2.
\end{equation}
\end{Thm}

\emph{Constants:}    
Note that the constant $b_2$ in \eqref{eq:delta-Th-4.2} and \eqref{eq:delta-Th-4.3} stems from Assumption \ref{ass:momentboundtail} and is kernel-dependent. Similarly, the last summand in \eqref{eq:delta-Th-4.3-1} depends on $\beta_1,\beta_2$ which are also kernel-dependent and stem from Assumption \ref{ass:eigenfunctiongrowth}. One of the key proof ideas of Theorem \ref{th:concentration_quafratic_form} is to separate the infinite Mercer representation into a finite truncation and a remainder term. The remainder term is then bounded by the last summand in \eqref{eq:delta-Th-4.3-1}. The choice of $M(T)$ in Theorem \ref{th:concentration_quafratic_form} also depends on the kernel. For instance, the Gaussian kernel requires $M(T) = \log(T)$ while the periodic kernel requires  $M(T) = T$; in this case, $b_2 =0$ in \eqref{eq:delta-Th-4.3}, so the rate is not effected.

\textit{Rate of convergence:}
We note that Theorems \ref{th:finiteMercer} and \ref{th:concentration_quafratic_form} achieve the optimal convergence rate $\sqrt{\log(T)/T}$ (potentially up to some logarithmic term). The rate is considered optimal in the sense that it coincides with the rate one would obtain if $\mathcal{K}(X,X)$ were replaced by a deterministic matrix with bounded eigenvalues.
Whether or not we get an additional logarithmic term depends on two aspects. 

The first aspect relates to the errors in the underlying regression problem \eqref{eq:nonlinearVAR_2}. 
If the errors in \eqref{eq:nonlinearVAR_2} are bounded, that is, they satisfy Assumption \ref{ass:boundednoise}, we can take $\gamma = L_{\varepsilon}$ in \eqref{eq:gamma-Th-4.2} and \eqref{eq:gamma-Th-4.3} and there is no additional logarithmic term caused by the errors.

The second aspect that can lead to an additional logarithmic term only applies 
when Assumption \ref{ass:kernelbound} is satisfied with $M=\infty$, namely when the kernels admit an infinite Mercer representation. Then, for instance, the Gaussian kernel, satisfies Assumption \ref{ass:eigenfunctiongrowth} with $M = \log(T)$ and $b_2 = dp/2 -1$ (see Lemma \ref{pro:sumbetasgaussian}). Whenever $dp > 2$, this results in an additional logarithmic term in the choices of $\delta$ in Theorems \ref{th:main_consistency} and \ref{th:concentration_quafratic_form}. 

\textit{Connection to U-statistics:}
We point out that a crucial assumption for Theorem \ref{th:concentration_quafratic_form} is the separability of the kernel (Assumption \ref{ass:separable_kernel_multi_parta}). In particular, there are certain U-statistics that may not satisfy a Mercer representation. A general result for U-statistics can be found in \cite{duchemin2023concentration}, who consider U-statistics of order two for Markov chains. Their results require a reverse Doeblin condition. While $\eta' \mathcal{K}(X,X) \eta $ in \eqref{eq:eta-quadratic-form} can be written as a U-statistic of a Markov chain, their results cannot be applied since the Markov chain does not satisfy the reverse Doeblin condition.

\subsection{Examples} \label{se:examples}
In this section we discuss several scenarios that are covered by our assumptions.

\subsubsection{Infinite Mercer representation} \label{se:example-infinite}
In this section, we present several examples of kernels that satisfy Assumption \ref{ass:separable_kernel_multi} with $M = \infty$ {and for which our results hold.}

\vspace{0.2cm}
\noindent
\textit{Gaussian kernel:}
Suppose the kernels $K_1,\dots,K_d$  are Gaussian, i.e.,
 \begin{equation}\label{eq:gaussiankerneldef1}
 K_i(x,y) = \exp\left( - {\| x - y\|^2 }/{\tau^2}\right), \quad x,y \in \clx^p.
 \end{equation}

Note that we use a common smoothing parameter $\tau$ for each kernel. Our results also hold if each kernel $K_i$ has a dimension-dependent smoothing parameter $\tau_i$. This change does not affect  the proofs, so we, in order to reduce the notational burden, assume that the kernels have a common smoothing parameter. 

Theorem \ref{thm:concentrationholdsforgaussian} says that the concentration inequality in Theorem \ref{th:concentration_quafratic_form} holds for the Gaussian kernel when the noise sequence is Gaussian. 

\begin{Thm}\label{thm:concentrationholdsforgaussian} 
   Suppose Assumption \ref{ass:gaussiannoise}. Then, under Assumption \ref{ass:gbounded}, the results from Theorems \ref{th:main_consistency} and   \ref{th:concentration_quafratic_form} hold for the Gaussian kernel.
\end{Thm}

\begin{proof}
     The result follows from Lemma \ref{le:Gaussian-Mercer-rep}, Lemma \ref{lem:gaussiankernelassumptiontailmoment}, and Lemma \ref{pro:sumbetasgaussian}.
\end{proof}

The proof of the results used in the proof of Theorem \ref{thm:concentrationholdsforgaussian} are quite involved and can be found in Section \ref{se:ProofsGaussianKernel}. We present two more examples but, for simplicity, consider the case $d = p=1$.

\vspace{0.2cm}
\noindent
\textit{Periodic kernel:}
Suppose the kernel $K_1$ is the periodic Sobolev kernel, i.e., 
\begin{equation} \label{eq:periodic-kernel}
\begin{split}
  K(x, y) &= \sigma^2 \bigg( 1 + 2 \sum_{k=1}^\infty \frac{1}{(2\pi k)^{2s}} \cos( 2\pi k(x - y)) \bigg)\\
  &= \sigma^2 \bigg( 1 + \frac{(-1)^{s+1}}{(2s)!} \mathrm{B}_{2s}( |x - y| ) \bigg), \quad x,y \in \RR,
  \end{split}
\end{equation}
where $\mathrm{B}_{2s}$ is the Bernoulli polynomial of degree $2s$; see Section 2.1 in \cite{wahba1990spline}. 

\begin{Thm}\label{thm:concentrationholdsforperiodic}
   Under Assumptions \ref{ass:gbounded} and \ref{ass:subgaussiannoise}, the result from Theorem  \ref{th:concentration_quafratic_form} holds for the periodic Sobolev kernel with $s>1/2$.
\end{Thm}

\begin{proof}
It suffices to verify Assumptions \ref{ass:separable_kernel_multi}--\ref{ass:momentboundtail}.
Assumption \ref{ass:separable_kernel_multi} holds by Section 5.3 in \cite{karvonen2023probabilistic}. Therein it is shown that, with 
\begin{equation}
    \phi_k(x) = \sigma \sqrt{2} (2\pi k)^{-s} \cos(2\pi k x), \quad 
    \phi_{-k}(x) = \sigma \sqrt{2} (2 \pi k)^{-s} \sin(2\pi k x), \quad k \in \NN,
\end{equation}
and  $\phi_0 \doteq  \sigma$, we have 
\begin{equation} \label{eq:periodic-kernel-Mercer}
  K(x, y) = 
  \sigma^2 \bigg( 1 + \frac{(-1)^{s+1}}{(2s)!} \mathrm{B}_{2s}( |x - y| ) \bigg) =
  \sum_{k \in \ZZ} \phi_k(x) \phi_k(y).
\end{equation}
Assumption \ref{ass:eigenfunctiongrowth} is satisfied since
\begin{equation}
\sum\nolimits_{k=1}^M \sum\nolimits_{j=1}^{N(k)} \beta^2_{i,j,k} 
=
\sum\nolimits_{k=1}^M \|\phi_{k}\|^2_{\infty}
\leq
\sum\nolimits_{k=1}^M \sigma^2 4 (2\pi k)^{-2s}
\le b_1.   
\end{equation}    
Finally, Assumption \ref{ass:momentboundtail} can be verified through
\begin{equation}
    \sum\nolimits_{j=1}^{N(k)}  \EE((\phi_{i,j,k}(X_{t-1}))^2) 
    =
    \EE(( \sigma \sqrt{2} (2\pi k)^{-s} \cos(2\pi k X_{t-1}))^2)
    \le
    c k^{-2s}
    = \alpha_k.
\end{equation}
Letting  $M(T)=T$,  we  have 
\begin{align}
    \sum\nolimits_{k=M(T)}^{\infty} \alpha_k 
    =
    \sum\nolimits_{k = T}^{\infty} c k^{-2s}
    \leq 
    \beta_1 T^{-\beta_2},
\end{align}
where, due to our choice of $s$,  $\beta_2 = 2s-1>0$. Analogous estimates hold for the terms involving $\phi_{-k}$, so the result follows.
\end{proof}

\subsubsection{Finite Mercer representation} \label{se:example-finite}

{In this section we provide two examples of kernels with finite Mercer representations to which our results can be applied.}

\vspace{0.2cm}
\noindent
\textit{Polynomial kernel:}
We assume that Assumption \ref{ass:boundednoise} is satisfied, and therefore that $\clx$ is compact. For $x,y \in \clx^p$ and $m \in \NN$, $c >0$, define the polynomial kernel 
\begin{equation}
    K(x,y) = (x'y + c)^m.
\end{equation}
 This kernel  satisfies  Assumption \ref{ass:separable_kernel_multi} with the  finite Mercer representation $K(x,y) = \phi(x)' \phi(y)$,
where $\phi(x) \in \RR^{\binom{dp+m}{m}}$ is given by 
\begin{equation}
\phi(x) =  \left( \left(\binom{m}{n_1,\dots,n_{dp+1}}\right)^{1/2} c^{n_{dp+1}/2} \prod\limits_{i=1}^{dp} x_i^{n_i} \right)_{n_i \in \NN_0, \; \sum\nolimits_{i=1}^{dp+1}n_i = m}.
\end{equation}
Since $\clx$ is compact, Assumption \ref{ass:eigenfunctiongrowth} is also satisfied.

\vspace{0.2cm}
\noindent
\textit{Mercer sigmoid kernel:}
The sigmoid kernel is defined using the sigmoid function  
\[
K_{\text{sig}}(x,y) = \tanh(a x'y + c), \quad x,y \in \RR^{dp}.
\]
The sigmoid function is widely used as an activation function in neural networks \citep{sharma2017activation}. However, $K_{\text{sig}}$ is not separable in the sense of Assumption \ref{ass:separable_kernel_multi}. Instead, \cite{carrington2014new} introduced the Mercer sigmoid kernel given by
    \begin{equation} \label{eq:sig-kernel}
        K(x,y) = \frac{1}{dp} \sum_{i = 1 }^{dp} 
        \tanh\left( \frac{x_i - c }{b} \right) \tanh\left( \frac{y_i - c }{b} \right), \quad x,y \in \clx^p.
    \end{equation}
and demonstrated that it exhibits similar behavior to $K_{\text{sig}}$. In addition, the Mercer sigmoid kernel defined in  \eqref{eq:sig-kernel} satisfies Assumptions \ref{ass:separable_kernel_multi} and \ref{ass:eigenfunctiongrowth}.

\subsubsection{Stochastic regression model} \label{se:stochastic regression}
As a special case of a nonlinear vector autoregression model, our results cover stochastic nonlinear regression models. Such regression models can be written as $y_t = g(X_t) + \varepsilon_t$,
where the $d$-dimensional predictors $X_t$ and the errors $\varepsilon_t$ are drawn independently from unknown distributions.

Then, Theorem \ref{th:main_consistency} is satisfied under the stated conditions. Those assumptions are satisfied, for instance, under the scenarios lined out in Sections \ref{se:example-infinite} and \ref{se:example-finite}. However, we note that for the Gaussian kernel, we need to impose additional assumptions on the distribution of the predictors. For example, our results hold when the predictors $X_t$ are drawn from a Gaussian distribution.

We note that while our results cover the stochastic regression case, this setting allows for an alternative proof of Theorem \ref{th:concentration_quafratic_form}. To be more precise, one can replace Theorem \ref{th:concentration_quafratic_form} by a conditional Hanson-Wright inequality.

\section{Auxiliary results} \label{se:auxiliary results}
For the reader's convenience, we state some of the results used throughout the paper here.

\subsection{Reproducing kernel Hilbert spaces} \label{se:repHil}

Suppose that $K_1,\dots,K_d : \clx^p \times \clx^p \to \RR$ are Mercer kernels with associated RKHS $\mathbb{H}_1,\dots,\mathbb{H}_d$, respectively. Let $\mathbb{H}$  be the Hilbert space defined as 
\begin{equation} \label{def:RKHS-multi}
\bbH \doteq \bigtimes\nolimits_{i=1}^{d} \mathbb{H}_i = \{ h \in \clb(\clx^p : \RR^d) :  h =  (h_1,\dots,h_d)', \;  h_i \in \mathbb{H}_i, \; i = 1,\dots,d\},
\end{equation}
with inner product 
\begin{equation}\label{eq:rkhsmultidefinnerproduct}
\langle h, \widetilde{h} \rangle_{\bmH} = \sum\nolimits_{i=1}^{d} \langle h_i, \widetilde{h}_i \rangle_{\bmH_i}.
\end{equation}
For a kernel $K_i$ satisfying Assumption \ref{ass:separable_kernel_multi}, we have, for $f,g \in \bbH_i$, 
\begin{equation*}
    \begin{split}
\langle f, g \rangle_{\bbH_i} = \sum\nolimits_{k=1}^{M} \lambda_{i,k}^{-1} \sum\nolimits_{j=1}^{N(k)} \langle f,  \phi_{i,j,k}\rangle_{L^{2}(\clx^p,\pi)}\langle g,  \phi_{i,j,k}\rangle_{L^{2}(\clx^p,\pi)},
\end{split}
\end{equation*}
where $\langle \cdot , \cdot \rangle_{L^{2}(\clx^p,\pi)}$ is the standard inner product on $L^{2}(\clx^p,\pi)$.

\begin{Def}
    A vector-valued RKHS is a Hilbert space $\bbH$ of functions $f: \mathcal{X}^p \to \RR^d$, such that for every $z \in \RR^d$, and 
 $x \in \clx^p$, the map $y \mapsto \clk(x,y)z$  belongs to $\bbH$ and, moreover, $\clk$ has the reproducing property
\begin{equation*}
    \langle f , \clk(x,\cdot) z \rangle_{\bbH} = f(x)'z.
\end{equation*}
\end{Def}

A discussion of vector-valued RKHSs can be found in \cite{alvarez2012kernels}. The following lemma is an immediate consequence of the fact that each $\bbH_i$ is an RKHS, \eqref{def:RKHS-multi}, and \eqref{eq:rkhsmultidefinnerproduct}.

\begin{Lem} \label{prop:RKHS-multi}
The space $\mathbb{H}$ is an RKHS with associated reproducing kernel \eqref{eq:diagkernel1}.
\end{Lem}

The following lemma shows that $\clk$ is uniformly bounded. 

\begin{Lem}\label{prop:clkkernelbasicprops}
Consider the kernel $\clk : \clx^p \times \clx^p \to \RR^{d\times d}$ defined in \eqref{eq:diagkernel1}. 
Under Assumption \ref{ass:kernelbound}, we have $\sup\nolimits_{x \in \clx^p} \| \clk(x, x) \| \le \kappa$.
\end{Lem}

\begin{proof}
    For $x \in \clx^p$, $\| \clk(x, x) \|^2 = \sum\nolimits_{i=1}^{d} | K_i(x,x)|^2 \le \sum\nolimits_{i=1}^{d} \kappa_i^2 = \kappa^2$. 
\end{proof}

The following result shows how the RKHS norm  can be used to bound the supremum norm of elements of $\bbH$. 

\begin{Lem}\label{lem:boundinfinitybyHnorm}
    If $f = (f_1,\dots,f_d)' \in \bbH$, then $\| f \|_{\infty} \le \kappa  \| f\|_{\bbH}$.
\end{Lem}
\begin{proof}
    From Corollary 4.36 of \cite{steinwart2008support}, we see that for each $i = 1,\dots,d$,
$\|f_i\|_{\infty} \le \kappa_i  \| f_i \|_{\bbH_i}$,
   so, using Cauchy-Schwarz inequality, 
    \begin{align*}
        \| f\|_{\infty} &\le 
        \max_{i=1,\dots,d} \|f_i\|_{\infty} 
        \le  
        \sum\limits_{i=1}^{d} \kappa_i  \| f_i \|_{\bbH_i}  \le \left( \sum\limits_{i=1}^{d} \kappa_i^2\right)^{1/2} \left(\sum\limits_{i=1}^{d} \| f_i\|^2_{\bbH_i}\right)^{1/2}
        = \left( \sum\limits_{i=1}^{d} \kappa_i^2\right)^{1/2} \| f\|_{\bbH}.
    \end{align*}
\end{proof}

\begin{Lem} \label{prob:ind_orthogonal}
Let $\bbH$ be a separable Hilbert space with inner product $\langle \cdot, \cdot \rangle_{\bbH}$. Let $X$ and $Y$ be $\bbH$-valued random variables such that $\EE_{\bbH}(X), \EE_{\bbH}(Y), \Var_{\bbH}(X)$, and $\Var_{\bbH}(Y)$ are well defined. If $X$ and $Y$ are independent, then 
\[
\EE\left[ \langle X , Y \rangle_{\bbH} \right] = \langle \EE_{\bbH}(X), \EE_{\bbH}(Y)\rangle_{\bbH}.
\]
\end{Lem}
Lemma \ref{prob:ind_orthogonal} follows from the discussion in Chapter 1.6 of \cite{bosq2000linear}.

\begin{Lem}\label{pro:multdimhilbertkernelquadproduct}
    Recall the definitions of $\clk_{X}$ and $\clk(X,X)$ from \eqref{eq:empkwithargument} and \eqref{eq:diagempkernelmatrix}, respectively. For $u,v \in \RR^{d(T-p)\times 1}$, we have 
    \[
    \langle \clk_X(\cdot) u, \clk_X(\cdot) v \rangle_{\bbH} = u' \clk(X,X) v.
    \]
\end{Lem}
\begin{proof}
    We write $
    u' = (u_1',\dots,u_d')$ and  $v' = (v_1',\dots,v_d')$,
    where $u_i' = (u_{i,p+1},\dots,u_{i,T})$ and $ v_i' = (v_{i,p+1},\dots,v_{i,T})$. Observe that 
    \begin{equation}\label{eq:kxufullkxvfull}
        \begin{split}
            \langle \clk_X(\cdot) u, \clk_X(\cdot) v \rangle_{\bbH} &= \sum\nolimits_{i=1}^d \langle (\clk_X(\cdot) u)_i, (\clk_X(\cdot) v)_i \rangle_{\bbH_i}\\
        \end{split}
    \end{equation}
    and
    \begin{equation}\label{eq:innermultikxukxv}
        \begin{split}
            \langle (\clk_X(\cdot) u)_i, (\clk_X(\cdot) v)_i \rangle_{\bbH_i} &= \sum\nolimits_{j=p+1}^{T} \sum\nolimits_{l=p+1}^{T}\langle (\clk_X(\cdot) )_{i,j}u_{i,j}, (\clk_X(\cdot))_{i,l} v_{i,l} \rangle_{\bbH_i}\\
            &= \sum\nolimits_{j=p+1}^{T} \sum\nolimits_{l=p+1}^{T} u_{i,j}\langle (\clk_X(\cdot) )_{i,j}, (\clk_X(\cdot))_{i,l} \rangle_{\bbH_i}v_{i,l} \\
             &= \sum\nolimits_{j=p+1}^{T} \sum\nolimits_{l=p+1}^{T} u_{i,j} K_i( X_{(j-p):(j-1)}, X_{(l-p):(l-1)}) v_{i,l} \\
             &= u'_i K_i(X,X) v_i.
        \end{split}
    \end{equation}
     The result follows on combining \eqref{eq:diagempkernelmatrix}, \eqref{eq:kxufullkxvfull}, and \eqref{eq:innermultikxukxv}.
\end{proof}

\begin{Lem} \label{prop:operator-emp-eq}
    For any bounded $f \in L^{2,d}(\clx^p, \pi)$, we have 
    \[
    (L_{\calK, X}  + \lambda I )^{-1} L_{\calK, X} f(z) = \calK_X(z) ( \calK(X,X) + \lambda T I_{d(T-p)})^{-1} G_f(X), \quad z \in \clx^p,
    \]
    where $G_f(X)$ is defined analogously to $G_g(X)$ in \eqref{eq:vectorizedregressionproblem1}.
\end{Lem}

\begin{proof}
Let $
h(z) \doteq ( L_{\calK,X } + \lambda I)^{-1} L_{\calK,X} f(z)$, 
so that
\begin{equation}\label{eq:hinvers831a}
(L_{\calK,X}+ \lambda I) h(z) = L_{\calK,X}f(z).
\end{equation}
Then, $h(z) = \lambda^{-1} ( - L_{\calK,X} h(z) + L_{\calK,X} f(z))$,
so there are $\nu_{p+1},\dots, \nu_{T} \in \RR^{d}$ such that with  
\[
\widehat{\nu} \doteq \begin{pmatrix}
    \widehat{\nu}_{1}' & \cdots 
 & \widehat{\nu}_{d}'
\end{pmatrix}' \in \RR^{d (T-p)}, \quad \widehat{\nu}_{i} = \begin{pmatrix}
    \nu_{i,p+1} &
    \nu_{i,p+2}& 
    \cdots  & 
    \nu_{i,T}
\end{pmatrix}' \in \RR^{T-p}, \quad i = 1,\dots,d,
\]
we have
\begin{equation}\label{eq:hznurep1}
h(z) = \sum\nolimits_{t=p+1}^{T} \calK(Y_{t-1}, z) \nu_t = \calK_X(z) \widehat{\nu}.
\end{equation}
It follows from \eqref{eq:hznurep1} and the definition of $L_{\calK,X}$ that
\begin{equation}\label{eq:hinvers964a}
    \begin{split}
        L_{\calK,X} h(z) &= \frac{1}{T}\sum\nolimits_{t=p+1}^T \mathcal{K}(Y_{t-1},z) h(Y_{t-1})\\
        &= \frac{1}{T} \sum\nolimits_{t=p+1}^{T} \calK(Y_{t-1},z) 
        \sum\nolimits_{s=p+1}^{T} \calK(Y_{s-1},Y_{t-1}) \nu_s
=  \calK_X(z) \frac{1}{T} \calK(X,X) \widehat{\nu}.
    \end{split}
\end{equation}
We also have that \begin{equation}\label{eq:LKX1-0}
\begin{split}
    L_{\mathcal{K},X}(f)(z) 
    &= 
    \frac{1}{T} \sum\nolimits_{t=p+1}^{T} \mathcal{K}(X_{(t-p):(t-1)}, z)f(X_{(t-p):(t-1)}) \\
    &= \frac{1}{T} \sum\nolimits_{t=p+1}^{T} \mathcal{K}(Y_{t-1}, z) f(Y_{t-1})
 =  \frac{1}{T} \mathcal{K}_X(z) G_f(X).
    \end{split}
\end{equation}
Combining \eqref{eq:hinvers831a}, \eqref{eq:hinvers964a}, and \eqref{eq:LKX1-0},  we see that
\begin{equation}\label{eq:979l2}
\calK_X(z)\left( \left( \calK(X,X) + \lambda T I_{d(T-p)} \right) \widehat{\nu} - G_f(X)\right) = 0, \quad z \in \clx^p.
\end{equation}
Recalling that $\calK$ is positive definite, it follows from the fact that \eqref{eq:979l2} holds for all $z \in \clx^p$ that 
\[
\widehat{\nu} = \left(\calK(X,X) + \lambda TI_{d(T-p)}\right)^{-1}G_f(X).
\]
Together, the previous display and \eqref{eq:hznurep1} imply that 
\[
h(z) = \calK_X(z) \left(\calK(X,X) + \lambda TI_{d(T-p)}\right)^{-1}G_f(X)
.
\]
The result follows.
\end{proof}

\subsection{Geometric ergodicity} \label{se:geom-erg} 

Our concentration results crucially depend on the geometric ergodicity of the process $\{Y_t\}$ defined in \eqref{eq:defyn1b}. We rephrase here a few properties of nonlinear VAR models. We begin by recalling the definition of a geometrically ergodic Markov chain.

\begin{Def}
An $\clx^p$-valued Markov chain $\{Y_t\}$ is said to be geometrically ergodic if
there is a $\pi \in \clp(\clx^p)$, a $\rho \in (0,1)$, and a $\pi$-integrable measurable function $J : \clx^p \to [0,\infty)$ such that
    \begin{equation*}
        \| P^n(y,\cdot) - \pi(\cdot) \|_{TV} \leq \rho^n J(y), \hspace{0.2cm} y \in \clx^p, \; n \in \NN_0,
    \end{equation*} 
where 
\begin{equation}\label{eq:ynkernel1}
P^n(y,A) \doteq \PP_y(Y_n \in A), \quad A \in \clb(\clx^p), \; y \in \clx^p,
\end{equation}
and $\| \cdot \|_{TV}$ denotes the total variation norm on $\clp(\clx^p)$.
\end{Def}

\begin{Rem}
We require the somewhat nonstandard assumption that $J$ is integrable. In view of \eqref{eq:nonlinearVAR_2} and Assumption \ref{ass:gbounded}, this is satisfied, for instance, when Assumption \ref{ass:subgaussiannoise} holds. Then, one can infer that the chain is aperiodic and that Doeblin condition is satisfied, which implies the uniform ergodicity of the chain; see Theorem 16.0.2 in \cite{meyn2012markov}.
\end{Rem}

The following lemma says that  the lagged process $\{Y_t\}$ is geometrically ergodic on $\clx^p$. This follows immediately from Theorem 1 of \cite{lu2001l1}  (see also Remark 4.1 therein).
\begin{Lem}\label{prop:geomergodic}
  Under Assumptions \ref{ass:gbounded} and \ref{ass:subgaussiannoise}, the Markov chain  $\{Y_t\}$ defined in \eqref{eq:defyn1b} is geometrically ergodic.
\end{Lem}

There have been several works discussing geometric ergodicity for nonlinear VAR models, including \cite{cline1999geometric} and \cite{lu2001l1}. Throughout this work, we take advantage of this property.

\subsection{Abstract estimates} 
The convergence rates for the KRR estimator $\widehat{g}_T$ established in Theorem \ref{th:main_consistency} depend on the deterministic quantity  $\|g_{\lambda} - g\|_{\infty}$. Theorem \ref{thm:dev.deterministic.bound} presents one estimate that can be used to control this term.

\begin{Thm}\label{thm:dev.deterministic.bound} 
Suppose the function $g= (g_1,\dots,g_d)'$ is continuous and Assumption \ref{ass:boundednoise}. If $L_{\mathcal{K}}^{-r}g\in L^{2,d}(\clx^p, \pi )$ for some $r \in (1/2, 1]$, then 
\begin{equation} 
\| g_\lambda- g \|_{\infty} \leq \kappa \lambda^{r-1/2} \|L_{\calK}^{-r}g \|_{L^{2,d}(\clx^p,\pi)}. 
\end{equation} 
\end{Thm}

The condition $L_{\calK}^{-r}g\in L^{2,d}(\clx^p, \pi )$ originates from \cite{smale2005shannon}, where $r$ is interpreted as a smoothness parameter for $g$. When $r=1/2$, the condition $L_{\calK}^{-1/2} g \in L^{2,d}(\clx^p, \pi )$ is equivalent to the assumption that $g \in\bbH$, as shown by the identity 
\begin{align*}
    \|L_{\calK}^{-1/2}g\|^2_{L^{2,d}(\clx^p,\pi)} 
&= \sum\nolimits_{i=1}^d \left\|\sum\nolimits_{k=1}^{\infty} \sum\nolimits_{j=1}^{N(k)} \langle g_{i}, \phi_{i,j,k} \rangle_{L^{2}(\clx^p,\pi)} \phi_{i,j,k}/\sqrt{\lambda_{i,k}}\right\|^2_{L^{2}(\clx^p,\pi)} 
\\&= \sum\nolimits_{i=1}^d \sum\nolimits_{k=1}^{\infty} \sum\nolimits_{j=1}^{N(k)} \langle g_{i}, \phi_{i,j,k} \rangle^2_{L^{2}(\clx^p,\pi)}/\lambda_{i,k} = \|g\|^2_\bbH. 
\end{align*}
Together with Theorem \ref{th:main_consistency}, Theorem \ref{thm:dev.deterministic.bound} provides a convergence rate for our estimator.  A proof of Theorem \ref{thm:dev.deterministic.bound} can be found in \cite{liu2023estimation}. We also refer the reader to Theorems 6(b) and 7 in \cite{liu2023estimation}, where the authors assume that the respective kernel admits an  eigendecomposition with respect to the Fourier basis, resulting in 
estimates similar to those in Theorem \ref{thm:dev.deterministic.bound}.

\section{Proofs of main results} \label{se:proofs}

\subsection{Proofs of results in Section \ref{se:main-results}} \label{se:proofs-main-results} 

Recall the definition of the integral operator $L_{\mathcal{K}}$ in \eqref{eq:lkiintegralop2}. We consider a sample analogue of $L_{\clk}$ defined by
\begin{equation}
\begin{split}
    L_{\mathcal{K},X}(f)(z) 
    &= 
    \frac{1}{T} \sum\nolimits_{t=p+1}^{T} \mathcal{K}(X_{(t-p):(t-1)}, z)f(X_{(t-p):(t-1)}) \\
    &= \frac{1}{T} \sum\nolimits_{t=p+1}^{T} \mathcal{K}(Y_{t-1}, z) f(Y_{t-1})
    \\
    &=  \frac{1}{T} \mathcal{K}_X(z) G_f(X), \quad z \in \clx^p,
    \end{split}\label{eq:LKX1}
\end{equation}
 where $G_f(X)$ is defined analogously to $G_g(X)$ in \eqref{eq:vectorizedregressionproblem1}. Similarly, recall the function $g_{\lambda}$ defined in \eqref{eq:glamdef1}. We consider a sample analogue $g_{X,\lambda} : \clx^p \to \RR^d$ given by  
\begin{equation}\label{eq:gxlamdef1}
    g_{X,\lambda} \doteq (L_{\calK, X} + \lambda I)^{-1} L_{\calK, X} g.
\end{equation}

%\begin{proof}[Proof of Theorem \ref{th:main_consistency}]
\noindent\textbf{Proof of Theorem \ref{th:main_consistency}}
Let $\delta$ satisfy \eqref{eq:delta-Th-4.1} and note that
\begin{equation} 
    \begin{aligned}
    \Prob\left(
    \|\widehat{g}_T - g \|_{\infty} > \delta \right)
    &\leq
    \Prob\left( \| \widehat{g}_T - g_{X,\lambda} \|_{\infty} > \delta/2 \right) +
    \Prob\left( 
    \|g_{X,\lambda}-g_\lambda\|_{\infty} + \| g_{\lambda} - g\|_{\infty} > \delta/2 \right),
\end{aligned}\label{eq:proofth4_al1}
\end{equation}
where $g_{\lambda}$ and $g_{X, \lambda}$ are defined in \eqref{eq:glamdef1} and \eqref{eq:gxlamdef1}, respectively. 
We consider the two summands in \eqref{eq:proofth4_al1} separately. 

Recall $L_{\calK, X}$ from \eqref{eq:LKX1}. Beginning with the second summand in \eqref{eq:proofth4_al1}, let 
\begin{align*}
    \mathcal{G}(X)&\doteq (L_{\calK, X} + \lambda I)^{-1} L_{\calK, X}( g) - (L_{\clk} + \lambda I)^{-1} L_{\clk} (g) 
    =  g_{X, \lambda} - g_{\lambda}.
    \end{align*}
With $\delta$ as in \eqref{eq:delta-Th-4.1}, we get
\begin{equation}\label{eq:thm4.1secondsummandterm}
\begin{split}
     \Prob\left( 
    \|g_{X,\lambda}-g_\lambda\|_\bbH + \| g_{\lambda} - g\|_{\infty} > \delta/2 \right) 
    &\leq
    \Prob\left( 
    \|g_{X,\lambda}-g_\lambda\|_\bbH > \frac{1}{\lambda} \sqrt{\frac{\log(T)}{T}} L^{\frac{b_2}{2}} \gamma \operatorname{C_{0}}(g) \right)
    \\&\leq 
    \Prob\left( 
    \|g_{X,\lambda}-g_\lambda\|_\bbH > \operatorname{C_{0}}(g) \frac{1}{\lambda} \sqrt{\frac{\log(T)}{T}} \right)
    \\&= 
    \PP\left( \| \mathcal{G}(X)\|_{\bbH} > \operatorname{C_{0}}(g) \frac{1}{\lambda} \sqrt{\frac{\log(T)}{T}} \right) \le c_1T^{-c_2},
\end{split}
\end{equation}
where the final inequality 
 is due to Lemma \ref{lem:newproximalvstrueregfunction}.

We now consider the first summand in \eqref{eq:proofth4_al1}. Using  \eqref{eq:vectorizedregressionproblem1}, \eqref{eq:ghatkernelrep1}, \eqref{eq:gxlamdef1}, and Lemma \ref{prop:operator-emp-eq}, we get 
\begin{align}
\widehat{g}_T(\cdot) - g_{X, \lambda}(\cdot) &= 
\clk_X(\cdot)  ( \clk(X,X) + \lambda T I_{d(T-p)})^{-1}Y - (L_{\clk,X}+\lambda I)^{-1}L_{\clk,X}(g)(\cdot) \nonumber\\
&= \clk_X(\cdot)  ( \clk(X,X) + \lambda T I_{d(T-p)})^{-1}Y - \clk_X(\cdot) (\clk(X,X) + \lambda T I_{d(T-p)})^{-1} G_g(X) \nonumber \\
&= \frac{1}{T} \clk_X(\cdot) ( \clk(X,X)/T + \lambda I_{d(T-p)})^{-1} \eta.\label{eq:ghattminusgxlambddef}
\end{align}
Let
\begin{align*}
Z_{X,T} &\doteq ( \clk(X,X)/T + \lambda I_{d(T-p)})^{-1} \\
&= \diag( (K_1(X,X)/T + \lambda I_{d(T-p)})^{-1}, \dots, (K_d(X,X)/T + \lambda I_{d(T-p)})^{-1}),
\end{align*}
and  observe that 
\begin{align}
\|\widehat{g}_T - g_{X,\lambda} \|^2_\bbH
&=
\frac{1}{T^2} \eta' ( \mathcal{K}(X,X)/T + \lambda I_{d(T-p)})^{-1} \clk(X,X) ( \clk(X,X)/T + \lambda I_{d(T-p)})^{-1} \eta
\label{eq:norm_ineq_1}
\\&=
\frac{1}{T^2} 
\tr( \eta' Z_{X,T} \mathcal{K}(X,X) Z_{X,T} \eta )
=
\frac{1}{T^2} 
\tr( \eta \eta' Z_{X,T} \mathcal{K}(X,X) Z_{X,T}  )
\nonumber
\\&\leq
(\lambda_{\max}(Z_{X,T}))^2
\frac{1}{T^2} 
\tr( \eta' \mathcal{K}(X,X)  \eta  )
\label{eq:norm_ineq_2}
\\&\leq
\lambda^{-2}
\frac{1}{T^2} 
\eta' \mathcal{K}(X,X)  \eta,
\label{eq:norm_ineq_3}
\end{align}
where the first identity follows from \eqref{eq:ghattminusgxlambddef} and Lemma \ref{pro:multdimhilbertkernelquadproduct}.
For the first inequality, note that the matrices $\eta \eta'$, $Z_{X,T}$ and $\mathcal{K}(X,X)$ are all positive semidefinite, so the inequality follows from Theorem 1 of \cite{fang1994inequalities}. The final inequality follows upon observing that
\begin{align*}
    \lambda_{\max}(Z_{X,T}) &= 
    \lambda_{\max}(( \mathcal{K}(X,X)/T + \lambda I_{d(T-p)})^{-1})\\
    &= (\lambda_{\min}( \mathcal{K}(X,X)/T + \lambda I_{d(T-p)}))^{-1}
    \leq \lambda^{-1}.
\end{align*}
It remains to show a high probability bound for {the quantity} in \eqref{eq:norm_ineq_3}. We distinguish between the cases where Assumption \ref{ass:separable_kernel_multi} is satisfied with $M<\infty$ or $M=\infty$.

\begin{itemize} \item \textit{$M<\infty$}: In this case, $\delta$ satisfies \eqref{eq:delta-Th-4.1} with $L\doteq M$.  Note that  $\lambda \delta$ satisfies the inequality in \eqref{eq:delta-Th-4.2}. Thus, we can apply Theorem \ref{th:finiteMercer}, which, under Assumptions \ref{ass:separable_kernel_multi} and \ref{ass:eigenfunctiongrowth}, ensures that there are constants $c_1,c_2 \in (0,\infty)$ such that for all $T$, with probability at least $1- c_1 T^{-c_2}$, 
\begin{equation} \label{eq:tildeineaulity}
\frac{1}{T^2} 
\eta' \mathcal{K}(X,X)  \eta \leq (\lambda \delta
)^2.
\end{equation}
 Then, the result follows on combining \eqref{eq:proofth4_al1}, \eqref{eq:thm4.1secondsummandterm}, \eqref{eq:norm_ineq_3}, and the probability bound in \eqref{eq:tildeineaulity}.

\item \textit{$M=\infty$}: 
In this case, $\delta$ satisfies \eqref{eq:delta-Th-4.1} with $L=M(T)$ as given in Assumption \ref{ass:momentboundtail}. Note that  $\lambda \delta$ satisfies the inequality in \eqref{eq:delta-Th-4.3}.
Thus, we can apply Theorem \ref{th:concentration_quafratic_form}, which, under Assumptions  \ref{ass:separable_kernel_multi}--\ref{ass:momentboundtail}, ensures that there are constants $c_1,c_2 \in (0,\infty)$ such that for all $T$, with probability at least $1- c_1 T^{-c_2}$, 
\begin{equation} \label{eq:tildeineaulity2}
\frac{1}{T^2} 
\eta' \mathcal{K}(X,X)  \eta \leq    (\lambda\delta)^2.
\end{equation}
Then, the result follows on combining \eqref{eq:proofth4_al1}, \eqref{eq:thm4.1secondsummandterm}, \eqref{eq:norm_ineq_3}, and the probability bound in \eqref{eq:tildeineaulity2}.
\end{itemize}
\vspace{-0.4cm}
\hfill $\blacksquare$
%\end{proof}

\noindent\textbf{Proof of Theorem \ref{th:finiteMercer}}
Note that
\begin{align}
    &
    \Prob \left( \eta' \mathcal{K}(X,X) \eta > \delta^2 \right)
    \nonumber \\
    &=
    \PP\left(  \frac{1}{T^2} 
    \sum\limits_{i=1}^{d}
\sum\limits_{t=p+1}^{T}\sum\limits_{s=p+1}^{T}  \varepsilon_{i,t} K_i(Y_{t-1},Y_{s-1})\varepsilon_{i,s}  > \delta^2\right) 
\nonumber
    \\&= \PP\left(\frac{1}{T^2} \sum\limits_{i=1}^d \sum\limits_{k=1}^{M} \lambda_{i,k} \sum\limits_{j=1}^{N(k)}  \sum\limits_{s=p+1}^{T} \sum\limits_{t=p+1}^{T} \varepsilon_{i,s} \varepsilon_{i,t} \phi_{i,j,k}(Y_{t-1}) \phi_{i,j,k}(Y_{s-1}) > \delta^2   \right) \nonumber \\
    &= \PP \left( \sum\limits_{i=1}^d \sum\limits_{k=1}^{M} \lambda_{i,k} \sum\limits_{j=1}^{N(k)} \left( \frac{1}{T} \sum\limits_{t=p+1}^{T} \varepsilon_{i,t} \phi_{i,j,k}(Y_{t-1}) \right)^2 > \delta^2\right)
    \nonumber\\
    &\le \sum\limits_{i=1}^{d}  \PP \left(  \sum\limits_{k=1}^{M} \lambda_{i,k} \sum\limits_{j=1}^{N(k)} \left( \frac{1}{T} \sum\limits_{t=p+1}^{T} \varepsilon_{i,t} \phi_{i,j,k}(Y_{t-1}) \right)^2 > \frac{\delta^2}{d} \right).
    \label{eq:firstquadraticform1220}
    \end{align}
For $i = 1,\dots, d$, define
\[
A_{i,M} \doteq \sum\limits_{k=1}^{M} \lambda_{i,k} \sum\limits_{j=1}^{N(k)} \left( \frac{1}{T} \sum\limits_{t=p+1}^{T} \varepsilon_{i,t} \phi_{i,j,k}(Y_{t-1})\right)^2,
\]
and observe that
\begin{align}
    &\PP\left( A_{i,M} > \frac{\delta^2}{d}\right) \nonumber
    \\&=  
    \PP\left( \sum\limits_{k=1}^{M}  \sum\limits_{j=1}^{N(k)} \left( \frac{1}{T} \sum\limits_{t=p+1}^{T} \varepsilon_{i,t}  \lambda_{i,k}^{1/2} \phi_{i,j,k}(Y_{t-1})\right)^2 > \frac{\delta^2}{d}\right)  
    \nonumber     \\
    &\le \PP\left( \max\limits_{k=1,\dots,M} 
    \max_{j =1, \dots, N(k)} \left( \frac{1}{T} \sum\limits_{t=p+1}^{T} \varepsilon_{i,t} \beta_{i,j,k}^{-1}\lambda_{i,k}^{1/2} \phi_{i,j,k}(Y_{t-1})\right)^2  \sum\limits_{k=1}^{M-1} \sum\limits_{j=1}^{N(k)} \beta_{i,j,k}^2 > \frac{\delta^2}{d}\right)
    \nonumber
    \\&\le 
    \PP \left( \max\limits_{k=1,\dots,M}   
    \max_{j =1, \dots, N(k)} \left( \frac{1}{T} \sum\limits_{t=p+1}^{T} \varepsilon_{i,t}  
    \beta_{i,j,k}^{-1}\lambda_{i,k}^{1/2} \phi_{i,j,k}(Y_{t-1})\right)^2  > \delta_1^2 \right), \label{eq:a1mfinitedecompbound1}
\end{align}
where \eqref{eq:a1mfinitedecompbound1} follows by \eqref{eq:assumptionbetasdef-2} in Assumption \ref{ass:eigenfunctiongrowth} with
\begin{equation}\label{eq:delta1def1}
        \delta_1 \doteq \delta \left( d b_1 M^{b_2} \right)^{-1/2}.
    \end{equation}
Then, from \eqref{eq:a1mfinitedecompbound1}, we see that, for  $\gamma>0$, 
\begin{align}
        &\PP\left(A_{i,M} > \frac{\delta^2}{d}\right)  \nonumber
          \\&\leq \PP\left( \max\limits_{k=1,\dots,M-1}   
          \max_{j =1, \dots, N(k)} \left| \frac{1}{T} \sum\limits_{t=p+1}^{T} \varepsilon_{i,t}   \beta_{i,j,k}^{-1}\lambda_{i,k}^{1/2} \phi_{i,j,k}(Y_{t-1}) \right|   > \delta_1 \right) \nonumber \\
          &\le \PP\left( \max\limits_{k=1,\dots,M}   
          \max_{j =1, \dots, N(k)}\bm{1}_{\{\|\eta\|_{\infty} {\le} \gamma\}} \left| \frac{1}{T} \sum\limits_{t=p+1}^{T} \varepsilon_{i,t}   \beta_{i,j,k}^{-1}\lambda_{i,k}^{1/2} \phi_{i,j,k}(Y_{t-1}) \right|   > \delta_1 \right) \nonumber \\
          &\hspace{1 cm} + \Prob( \|\eta\|_{\infty} > \gamma ) \label{eq:paimdeltasqd1-1} \\
          &\leq
          \PP\left( \max\limits_{k=1,\dots,M}   
          \max_{j =1, \dots, N(k)} \left| \frac{1}{T} \sum\limits_{t=p+1}^{T} \varepsilon_{i,t}   \beta_{i,j,k}^{-1}\lambda_{i,k}^{1/2} \phi_{i,j,k}(Y_{t-1}) \bm{1}_{\{ | \varepsilon_{i,t} | \leq \gamma \}} \right|   > \delta_1 \right) \nonumber 
          \\
          &\hspace{1cm}+
          \Prob( \|\eta\|_{\infty} > \gamma ),
\label{eq:paimdeltasqd1}
\end{align}
  where \eqref{eq:paimdeltasqd1-1} is due to a truncation argument by intersecting with the event $\{ \|\eta\|_{\infty} \leq \gamma \}$, and \eqref{eq:paimdeltasqd1} follows on observing that for  $a_1,\dots, a_n,b_1, \dots, b_n,c \in \RR$,
\[
\bm{1}_{\{\max\{a_1,\dots,a_n\} \le c\}}\left| \sum\nolimits_{i=1}^{n} b_i\right| \le \left| \sum\nolimits_{i=1}^{n} \bm{1}_{\{ a_i \le c\}}b_i  \right|.
\] 
Letting, for $i = 1,\dots,d$, 
\begin{equation}\label{eq:P1imdef1}
P_{1,i,M} \doteq  \PP\left( \max\limits_{k=1,\dots,M}   
          \max_{j =1, \dots, N(k)} \left| \frac{1}{T} \sum\limits_{t=p+1}^{T} \varepsilon_{i,t}   \beta_{i,j,k}^{-1}\lambda_{i,k}^{1/2} \phi_{i,j,k}(Y_{t-1}) \bm{1}_{\{ | \varepsilon_{i,t} | \leq \gamma \}} \right|   > \delta_1 \right),
\end{equation}
and 
\begin{equation}\label{eq:P2imdef1}
P_{2,i,M} \doteq  \Prob( \|\eta\|_{\infty} > \gamma) ,
\end{equation}
it follows from \eqref{eq:paimdeltasqd1} that
\begin{equation}\label{eq:probaimdecomposition}
\PP\left(A_{i,M} > \frac{\delta^2}{d}\right)  \le P_{1,i,M} + P_{2,i,M}.
\end{equation}
We consider the  probabilities $P_{1,i,M}$ and $P_{2,i,M}$ separately.
For $P_{1,i,M}$ in \eqref{eq:P1imdef1}, consider the function 
$F_{i,M} : (\mathcal{X}^p)^{T-p+1} \to [0,\infty)$ defined by 
\begin{equation*}
\begin{split}
   & F_{i,M}(y_p,\dots,y_T)\\
   &= \max_{k=1,\dots,M} \max_{j =1, \dots, N(k)} \frac{1}{T}
    \left| \sum_{t=p+1}^{T}  \left((y_{t,1} - f(y_{t-1}))_i \beta_{i,j,k}^{-1} \lambda^{\frac{1}{2}}_{i,k} \phi_{i,j,k} (y_{t-1})
    \bm{1}_{\{ | (y_{t,1} - f(y_{t-1}))_i | \leq \gamma \}} \right)
    \right|.
\end{split}
\end{equation*}
Then, 
    \begin{equation}\label{eq:p1imnewrep}
    \begin{split}
    P_{1,i,M} &= \PP\left( F_{i,M}(Y_{p:T})  >  \delta_1 \right) =  \Prob\left(
    F_{i,M}(Y_{p:T})   -   \EE\left( F_{i,M}(Y_{p:T})\right)
    > \delta_1 - \EE\left( F_{i,M}(Y_{p:T})\right)
    \right).
    \end{split}
    \end{equation}
    To establish an upper bound for the probability in  \eqref{eq:p1imnewrep}, we first estimate the expected value in that display. Observe that 
\begin{equation}\label{eq:expectedfimypt2}
\begin{split}
    &
    \left(\EE \left(F_{i,M}(Y_{p:T})\right)\right)^2
    \leq
    \EE \left(\left(F_{i,M}(Y_{p:T}) \right)^2\right)
    \\
    &=
    \EE\left( \max_{k=1,\dots,M} \max\limits_{j=1,\dots,N(k)}
    \left| \frac{1}{T} \sum_{t=p+1}^{T} \varepsilon_{i,t} \beta_{i,j,k}^{-1} \lambda^{\frac{1}{2}}_{i,k} \phi_{i,j,k}(Y_{t-1}) \bm{1}_{\{ | \varepsilon_{i,t} | \leq \gamma \}} \right|^2 \right) 
    \\
    &=
    \EE \left( \max_{k=1,\dots,M} \max\limits_{j=1,\dots,N(k)} \frac{1}{T^2} \sum_{s,t=p+1}^{T}
    \varepsilon_{i, s} \varepsilon_{i, t} \bm{1}_{\{ | \varepsilon_{i,s} |, | \varepsilon_{i,t} | \leq \gamma \}} \beta_{i,j,k}^{-2} \lambda_{i,k} \phi_{i,j,k} (Y_{t-1}) \phi_{i,j,k} (Y_{s-1})  \right)
    \\&\leq
    \frac{1}{T^2} \sum_{s,t=p+1}^{T}
    \EE \left( 
    \varepsilon_{i, s} \varepsilon_{i, t} \bm{1}_{\{ | \varepsilon_{i,s} |, | \varepsilon_{i,t} | \leq \gamma \}} 
    \max_{k=1,\dots,M} \max\limits_{j=1,\dots,N(k)} 
    \beta_{i,j,k}^{-2} \lambda_{i,k} \phi_{i,j,k} (Y_{t-1}) \phi_{i,j,k} (Y_{s-1})  \right),
\end{split}
\end{equation}
where the first inequality uses Jensen's inequality and the second inequality uses that the maximum of a sum is bounded above by the sum of the maxima. We consider the diagonal and cross terms of the  sum in the last line of \eqref{eq:expectedfimypt2} separately. Define
\begin{align}
    \Phi(Y_{t-1}, Y_{s-1}) \doteq
    \max_{k=1,\dots,M} \max\limits_{j=1,\dots,N(k)} 
    \beta_{i,j,k}^{-2} \lambda_{i,k} \phi_{i,j,k} (Y_{t-1}) \phi_{i,j,k} (Y_{s-1}),
\end{align}
and note, due to the definition of $\beta_{i,j,k}$ in \eqref{eq:assumptionbetasdef}, that $
        \| \Phi\|_{\infty} = 1$. For the cross terms in the last line of \eqref{eq:expectedfimypt2}, if  $t>s$, then  
\begin{align}
    &
    \EE \bigg( 
    \varepsilon_{i, s} \varepsilon_{i, t} \bm{1}_{\{ | \varepsilon_{i,s} |, | \varepsilon_{i,t} | \leq \gamma \}} 
    \Phi(Y_{t-1}, Y_{s-1}) \bigg)
    \nonumber
    \\&=
    \EE \bigg( 
    \EE \bigg[ 
    \varepsilon_{i, s} \varepsilon_{i, t} \bm{1}_{\{ | \varepsilon_{i,s} |, | \varepsilon_{i,t} | \leq \gamma \}} 
    \Phi(Y_{t-1}, Y_{s-1}) \mid 
    X_{1:(r-1)}, \; r= \max\{s,t\}
    \bigg] \bigg)
    \label{eq:expectedfx1tsquared+1}
    \\&=
    \EE \bigg( 
    \Phi(Y_{t-1}, Y_{s-1})  \varepsilon_{i, s}  \bm{1}_{\{|\varepsilon_{i,s}| \le \gamma\}}
    \EE \bigg[ 
   \varepsilon_{i, t} \bm{1}_{\{  | \varepsilon_{i,t} | \leq \gamma \}} 
    \mid 
    X_{1:(t-1)}
    \bigg] \bigg)
    \nonumber
    \\&=
    \EE \bigg( 
    \Phi(Y_{t-1}, Y_{s-1})
    \varepsilon_{i, s} \bm{1}_{\{ | \varepsilon_{i,s} | \leq \gamma \}}
    \bigg)
    \EE \big[ 
    \varepsilon_{i, t} \bm{1}_{\{ | \varepsilon_{i,t} | \leq \gamma \}} 
    \big]
    \label{eq:expectedfx1tsquared+2}
    \\&\leq
    \left(
    \EE ( 
    \varepsilon^2_{i, s} )
    \EE \big[ 
    \bm{1}_{\{ | \varepsilon_{i,s} | \leq \gamma \}} 
    \big]   \right)^{\frac{1}{2}} 
    \EE \big[ 
    \varepsilon_{i, t} \bm{1}_{\{ | \varepsilon_{i,t} | \leq \gamma \}} 
    \big]
    \label{eq:expectedfx1tsquared+3}
    \leq
    2\sigma^2 (dT)^{-1},
\end{align}
where \eqref{eq:expectedfx1tsquared+1} uses the law of total expectation and \eqref{eq:expectedfx1tsquared+2} uses that $\varepsilon_{t}$ is independent of $Y_1,\dots,Y_{t-1}$ and $\varepsilon_{1}, \dots, \varepsilon_{t-1}$. The first inequality in \eqref{eq:expectedfx1tsquared+3} uses the fact that $\|\Phi\|_{\infty} = 1$ and   the Cauchy-Schwarz inequality.    
The second inequality in \eqref{eq:expectedfx1tsquared+3} follows on observing that, since $\EE( \varepsilon_{i,t}) = 0$,  with $\gamma$ as in \eqref{eq:gamma-Th-4.2}, 
\begin{align}\label{eq:offidagonalPhi}
    \left( \EE \left( 
    \varepsilon_{i, t} \bm{1}_{\{ | \varepsilon_{i,t} | \leq \gamma \}} \right) \right)^2
    &=
    \left( - \EE \left( 
    \varepsilon_{i, t} \bm{1}_{\{ | \varepsilon_{i,t} | > \gamma \}} \right) \right)^2
    \leq
    \EE \left( 
    \varepsilon^2_{i, t} \right) 
    \EE \left( \bm{1}_{\{ | \varepsilon_{i,t} | > \gamma \}} \right)
    \\&
    \leq \sigma^2
    \Prob \left( | \varepsilon_{i,t} | > \gamma \right)
    \leq 2\sigma^2
    \exp\left( -\gamma^2/ (2\sigma^2) \right)
    \leq 2\sigma^2 (dT)^{-2},
\end{align}
where the third inequality in \eqref{eq:offidagonalPhi} follows from  Assumption \ref{ass:subgaussiannoise}.

For the diagonal terms of the sum in the last line of \eqref{eq:expectedfimypt2}, we have
\begin{align}\label{eq:diagtermsPhi}
    &
    \EE \bigg( 
    \EE \bigg[ 
    \varepsilon^2_{i, t} \bm{1}_{\{ | \varepsilon_{i,t} | \leq \gamma \}} 
    \Phi(Y_{t-1}, Y_{t-1}) \mid 
    X_{1:(t-1)}
    \bigg] \bigg)\\
&    =
    \EE \left( 
    \varepsilon^2_{i, t} \bm{1}_{\{ | \varepsilon_{i,t} | \leq \gamma \}} \right)
    \EE \bigg( 
    \Phi(Y_{t-1}, Y_{t-1}) \bigg)
    \leq
    \sigma^2,
\end{align}
where we once more used that $\|\Phi\|_{\infty} = 1$. Combining \eqref{eq:expectedfimypt2}, \eqref{eq:expectedfx1tsquared+3}, \eqref{eq:offidagonalPhi} and \eqref{eq:diagtermsPhi}, we see that
\begin{equation}\label{eq:expfimypt23sigmabound}
(\EE(F_{i,M}(Y_{p:T})))^2 \le 3\frac{\sigma^2}{T}.
\end{equation}

From \eqref{eq:p1imnewrep} and \eqref{eq:expfimypt23sigmabound}, we see that
\begin{equation}\label{eq:boundp1im1423}
\begin{split}
P_{1,i,M} &\le   \Prob\left(
    F_{i,M}(Y_{p:T})   -   \EE\left( F_{i,M}(Y_{p:T})\right)
    > \delta_{2}
    \right),
\end{split}
\end{equation}
where, with $\delta_1$ as in \eqref{eq:delta1def1}, 
\begin{equation}\label{eq:delta2delta1minussigmasqrtt}
\delta_2 \doteq \delta_1 - \frac{2 \sigma }{\sqrt{T}}.
\end{equation}
Note that Lemma \ref{pro:separetelybounded} says that $F_{i,M}$ is separately bounded by $ { 2 \gamma}/{T}$. Then, applying Theorem 0.2 in \cite{dedecker2015subGaussian} with $L =  { 2 \gamma}/{T}$,  we can see that there is a $C_{\operatorname{mc}} \in (0,\infty)$ such that 
\begin{equation*}\label{eq:finalprobp1imbound}
P_{1,i,M} \le 2 \exp \left( -\frac{ C_{\operatorname{mc}} \delta_2^2 }{T \left( { 2 \gamma}/{T}\right)^2}\right) = 2 \exp\left( - \frac{ C_{\operatorname{mc}}  T \delta_2^2}{4\gamma^2} \right).
\end{equation*}
It remains to bound the probability $P_{2,i,M}$ defined in  \eqref{eq:P2imdef1}. Using  Lemma \ref{pro:extremevalepsilons}, we see that 
\begin{equation}\label{eq:p2improbobound}
P_{2,i,M} \le \exp\left(- \left( \gamma - \sqrt{ 2 \sigma^2 \log(dT)} \right)^2/(2\sigma^2) \right).
\end{equation}
Together, \eqref{eq:probaimdecomposition},  \eqref{eq:boundp1im1423}, \eqref{eq:delta2delta1minussigmasqrtt}, and \eqref{eq:p2improbobound} show that
\begin{equation}\label{eq:paimupperbound1}
\begin{split}
    &
    \PP\left( A_{i,M} > \frac{\delta^2}{d} \right) 
    \\&\le 
    2 \exp\left( - \frac{ C_{\operatorname{mc}} T \delta_2^2}{4\gamma^2} \right)
    +
    \exp\left(- \left( \gamma - \sqrt{2 \sigma^2 \log(dT)} \right)^2/(2\sigma^2) \right),
    \\&\le 
    2 \exp\left( - C_{\operatorname{mc}}  \frac{T}{\gamma^2} \left( \delta \left( d b_1 M^{b_2} \right)^{-1/2} - \frac{\sigma }{\sqrt{T}} \right)^2 \right)
    +
    \exp\left(- \left( \gamma - \sqrt{ 2 \sigma^2 \log(dT)} \right)^2/(2\sigma^2) \right)
    \\&\leq
    \exp\left(- C_{\operatorname{mc}} c_{1,T}(\delta, \gamma, M) \right)
    +
    \exp\left(- c_{2,T}(\gamma) \right),
\end{split}
\end{equation}
where $c_{1,T}, c_{2,T}$ are as in \eqref{eq:functions-g1-g2}. Together, \eqref{eq:firstquadraticform1220}  and \eqref{eq:paimupperbound1} complete the proof. 
\hfill $\blacksquare$

\vspace{0.5cm}
\noindent
\textbf{Proof of Theorem \ref{th:concentration_quafratic_form}}
In the following, we write $\eta' \mathcal{K}(X,X) \eta$ as a sum of a kernel with a finite series expansion and a remainder term. We then apply Theorem \ref{th:finiteMercer} to the part of the decomposition corresponding to a kernel with a finite expansion. The remainder term is handled using suitable upper bounds for the second moments of the eigenfunctions in the tail of the series expansion. Using arguments similar to those in \eqref{eq:firstquadraticform1220}, we can decompose the probability of interest as
\begin{align}
    \Prob \left( \eta' \mathcal{K}(X,X) \eta > \delta^2 \right)
    &= \PP \left( \sum\limits_{i=1}^d \sum\limits_{k=1}^{\infty} \lambda_{i,k} \sum\limits_{j=1}^{N(k)} \left( \frac{1}{T} \sum\limits_{t=p+1}^{T} \varepsilon_{i,t} \phi_{i,j,k}(Y_{t-1}) \right)^2 > \delta^2\right)
    \nonumber\\
    &\le \sum\limits_{i=1}^{d}  \PP \left(  \sum\limits_{k=1}^{\infty} \lambda_{i,k} \sum\limits_{j=1}^{N(k)} \left( \frac{1}{T} \sum\limits_{t=p+1}^{T} \varepsilon_{i,t} \phi_{i,j,k}(Y_{t-1}) \right)^2 > \frac{\delta^2}{d} \right)
    \nonumber\\
    &\le \sum\limits_{i=1}^{d} \left( \PP \left( A_{i,M}  > \frac{\delta^2}{2d} \right) + \PP\left( B_{i,M} > \frac{\delta^2}{2d} \right)\right),
    \label{eq:use_sep_3}
    \end{align}
where
    \begin{equation*}
    A_{i,M} = \sum\nolimits_{k=1}^{M-1} \lambda_{i,k} \sum\nolimits_{j=1}^{N(k)} \left( \frac{1}{T} \sum\nolimits_{t=p+1}^{T} \varepsilon_{i,t} \phi_{i,j,k}(Y_{t-1})\right)^2, 
    \end{equation*}
    and
    \begin{equation*}
    B_{i,M}  = \sum\nolimits_{k=M}^{\infty} \lambda_{i,k} \sum\nolimits_{j=1}^{N(k)} \left( \frac{1}{T} \sum\nolimits_{t=p+1}^{T} \varepsilon_{i,t} \phi_{i,j,k}(Y_{t-1})\right)^2.
\end{equation*}
We consider the two probabilities in \eqref{eq:use_sep_3} separately. For the first summand, using Theorem \ref{th:finiteMercer}, we  see that\begin{equation}\label{eq:applyfinitemercerlem}
\begin{split}
  \sum\nolimits_{i=1}^{d}    \PP \left( A_{i,M}  > \frac{\delta^2}{2d} \right) &\leq 
    d\exp\left(- C_{\operatorname{mc}} c_{1,T}(\delta, \gamma, M) \right)
    +
    d\exp\left(- c_{2,T}(\gamma) \right).
    \end{split}
\end{equation}
For the second summand in \eqref{eq:use_sep_3}, since $B_{i,M} \ge 0$, using Markov's inequality gives
\begin{align} \label{eq:probdboundbmarkov}
    \PP\left(B_{i,M} > \frac{\delta^2}{2d}\right) &\le \frac{2d}{\delta^2} \EE(B_{i,M}).
\end{align}
Applying the Fubini-Tonelli Theorem, we have
\begin{align}
    \EE(B_{i,M}) 
    &= 
    \EE \left( \sum\limits_{k=M}^{\infty} \lambda_{i,k} \sum\limits_{j=1}^{N(k)} \left( \frac{1}{T} \sum\limits_{t=p+1}^{T} \varepsilon_{i,t} \phi_{i,j,k}(Y_{t-1})\right)^2 \right) \nonumber \\
    &\leq \frac{\sigma^2}{T^2} \sum\limits_{k=M}^{\infty} \lambda_{i,k} \sum\limits_{j=1}^{N(k)} \sum\limits_{t=p+1}^{T} \EE( (\phi_{i,j,k}(Y_{t-1}))^2),
    \label{eq:estimatebm1a}
\end{align}
where \eqref{eq:estimatebm1a} follows upon using the law of total expectation and noting that $\varepsilon_{t}$ is independent of $Y_{t-1}$, that $\EE(\varepsilon_{i,t}) = 0$, that $\varepsilon_{i,t}$ and $\varepsilon_{i,s}$ are independent whenever $s \ne t$, and that $\EE(\varepsilon_{i,t}^2) \leq \sigma^2$, where $\sigma^2 \doteq \max_{i=1,\dots, d} \sigma_{i}^2$. Then, using Assumption  \ref{ass:momentboundtail}, it follows from \eqref{eq:estimatebm1a} that 
\begin{equation}\label{eq:ebimmomentfirst}
\begin{split}
    \EE(B_{i,M})
    &\le
    \frac{\sigma^2}{T^2} \sum\limits_{k=M}^{\infty} \lambda_{i,k}  \sum\limits_{j=1}^{N(k)} \sum\limits_{t=p+1}^{T} \EE( (\phi_{i,j,k}(Y_{t-1}))^2) 
    \le 
    \frac{\sigma^2}{T}  \sum\limits_{k=M(T)}^{\infty} \alpha_k
    \le 
    \frac{\sigma^2}{T}  \beta_1 T^{-\beta_2}.
\end{split}
\end{equation}
From \eqref{eq:probdboundbmarkov} and \eqref{eq:ebimmomentfirst}, we see that 
\begin{equation}\label{eq:bimfinalprobineq}
\PP\left( B_{i,M} > \frac{\delta^2}{2d}\right) \le  \frac{ 2 d \sigma^2}{ \delta^2  T} \beta_1 T^{-\beta_2}.
\end{equation}
Finally, on combining \eqref{eq:use_sep_3},  \eqref{eq:applyfinitemercerlem}, and \eqref{eq:bimfinalprobineq}, we see that
\[
\Prob \left( \eta' \mathcal{K}(X,X) \eta > \delta^2 \right)
\leq
d\exp\left(- C_{\operatorname{mc}} c_{1,T}(\delta, \gamma, M) \right)
+
d\exp\left(- c_{2,T}(\gamma) \right)
+
\frac{ 2 d^2 \sigma^2 }{ \delta^2 T} \beta_1 T^{-\beta_2},
\]
where $c_{1,T}$ and $c_{2,T}$ are as in  \eqref{eq:functions-g1-g2}.
\hfill $\blacksquare$

\subsection{Auxiliary concentration results and their proofs}

\begin{Lem}\label{lem:newproximalvstrueregfunction}
For the regression function $g : \clx^p \to \RR^d$ satisfying Assumption \ref{ass:gbounded}, let
\begin{align} \label{eq:Gbar}
\mathcal{G}(X) & \doteq (L_{\calK, X} + \lambda I)^{-1} L_{\calK, X} g - (L_\calK + \lambda I)^{-1} L_\calK  g .
\end{align}
There are constants $c_1,c_2 \in (0,\infty)$ and $\operatorname{c_0} \in (0,\infty)$ such that, if $\delta \geq \operatorname{C_{0}}(g) \frac{1}{\lambda} \sqrt{{\log(T)}/{T}}$ where  
$\operatorname{C_{0}}(g) \doteq  \operatorname{c_0} \sqrt{d} \kappa^2 \| g \|_{\infty}$, then, with probability at least $1-c_1T^{-c_2}$, we have 
$\|\mathcal{G}(X) \|_\bbH\leq \delta$.
\end{Lem}
\begin{proof}
Observe that
\begin{align*}
&\mathcal{G}(X)\\&= (L_{\calK, X} + \lambda I)^{-1} L_{\calK,X}g - (L_{\calK,X} + \lambda I)^{-1} L_{\calK} g +  (L_{\calK,X} + \lambda I)^{-1} L_{\calK} g- (L_{\calK} + \lambda I)^{-1} L_{\calK} g,
\end{align*}
so
\begin{equation}\label{eq:barGtriangleineq1}
\|\mathcal{G}(X)\|_{\bbH} \le \|  (L_{\calK, X} + \lambda I)^{-1} 
    \left( L_{\calK, X} g - L_{\calK}g \right) \|_{\bbH} +\|
    (L_{\calK, X} + \lambda I)^{-1}L_{\calK}g - 
    (L_\calK + \lambda I)^{-1} L_\calK  g\|_{\bbH}.
\end{equation}
We consider the two summands in \eqref{eq:barGtriangleineq1} separately.
First, since $L_{\mathcal{K},X}$ is self adjoint and positive semidefinite, we have
\begin{equation}\label{eq:lkxinverselambdaineq}
    \|  ( L_{\calK, X} + \lambda I )^{-1} \left( L_{\calK, X} g - L_{\calK}g \right)  \|_{\bbH}  \le \lambda^{-1} \|  L_{\calK, X} g - L_{\calK}g \|_{\bbH}
    \le \delta,
\end{equation}
where the last inequality holds for 
$\delta \geq \operatorname{c_0} \lambda^{-1} \sqrt{d} \kappa\| g \|_{\infty} \sqrt{{\log(T)}/{T}}$
with probability at least $1-c_1T^{-c_2}$, which follows from Lemma \ref{le:new_ana_LIU_15} since $g$ is bounded.
For the second summand in \eqref{eq:barGtriangleineq1}, we can once more use the fact that $L_{\calK,X}$ is positive semidefinite to see that
\begin{equation}\label{eq:lem6.3largedisplayinverselam}
\begin{split}
&
    \| (L_{\calK, X} + \lambda I)^{-1}L_{\calK}g - 
    (L_\calK + \lambda I)^{-1} L_\calK  g \|_\bbH  \\
&=  
    \|(L_{\calK, X}  + \lambda I)^{-1} (L_\calK + \lambda I) (L_\calK + \lambda I)^{-1} L_\calK g - (L_{\calK, X}  + \lambda I)^{-1}(L_{\calK, X}  + \lambda I)(L_\calK + \lambda I)^{-1} L_\calK  g \|_\bbH  \\
&=
    \|(L_{\calK, X}  + \lambda I)^{-1} \left(
    (L_\calK + \lambda I) (L_\calK + \lambda I)^{-1} L_\calK g - (L_{\calK, X}  + \lambda I)(L_\calK + \lambda I)^{-1} L_\calK  g \right) \|_\bbH\\
&\le  \lambda^{-1}     
    \left\|
    L_\calK \left((L_\calK + \lambda I)^{-1} L_\calK g\right) - 
    L_{\calK, X}\left((L_\calK + \lambda I)^{-1} L_\calK  g \right)\right\|_\bbH\\
    &= \lambda^{-1} \| L_{\calK} g_{\lambda}- L_{\calK,X} g_{\lambda} \|_{\bbH},
\end{split}
\end{equation}
where $g_{\lambda}$ is as in \eqref{eq:glamdef1}.  To apply Lemma \ref{le:new_ana_LIU_15}, we need to show that $g_{\lambda}$ is bounded. Taking $\widetilde{g} \doteq 0$ in \eqref{eq:glambdaoptimize}, we see that
\[
\| g_{\lambda} - g\|_2^2 + \lambda \|g_{\lambda}\|_{\bbH}^2 \le \| \widetilde{g} - g\|_2^2 + \lambda \| \widetilde{g}\|_{\bbH}^2 = \|g\|_2^2,
\]
which ensures that 
\[
\| g_{\lambda}\|_2^2 \le \sqrt{2} \| g\|_2, \quad \|g_{\lambda}\|_{\bbH} \le \lambda^{-1/2} \|g\|_{2}.
\]
Using Lemma \ref{lem:boundinfinitybyHnorm}, we see that 
\[
\|g_{\lambda}\|_{\infty} \le \kappa \|g_{\lambda}\|_{\bbH} \le \kappa \lambda^{-1/2}  \| g\|_2 \le \kappa \lambda^{-1/2}  \| g\|_\infty,
\]
which shows that $g_{\lambda}$ is bounded. 
Since $g_{\lambda}$ is bounded, it follows from Lemma \ref{le:new_ana_LIU_15} that  there is some  $\operatorname{c_0} > 0$ such that 
with 
$\delta \geq \operatorname{c_0} \sqrt{d} \kappa^2 \lambda^{-1/2} \| g \|_{\infty} \sqrt{{\log(T)}/{T}}$
there are  constants $c_1,c_2 \in (0,\infty)$ such that, with probability at least $1 - c_1 T^{-c_2}$,  
$\| L_{\calK,X}g_{\lambda} - L_{\calK}g_{\lambda} \|_{\bbH} \le  \delta$.
\end{proof}

\begin{Lem} \label{le:new_ana_LIU_15}
    Let $f = (f_1,\dots,f_d)': \clx^p \to \RR^d$ be a bounded map. There are  constants $c_1,c_2 \in (0,\infty)$ and $\operatorname{c_0}>0$ such that, if $\delta \geq \operatorname{c_0} \sqrt{d} \kappa\| f \|_{\infty} \sqrt{{\log(T)}/{T}}$, then, with probability at least $1 - c_1 T^{-c_2}$, 
    \[
    \| L_{\calK, X}f - L_{\calK}f\|_{\bbH} \le  \delta.
    \]
    \end{Lem}

\begin{proof} 
Let $F : (\clx^{p})^{T-p} \to \RR$ be given by 
    \begin{equation} \label{eq:fctF22}
       F(y_{p},\dots,y_{T-1}) = \left\| \frac{1}{T} \sum\limits_{t=p+1}^T f(y_{t-1}) \clk(y_{t-1},\cdot) - L_{\clk}f\right\|_{\bbH}, \quad y_p,\dots,y_{T-1} \in \clx^p,
        \end{equation}
    and observe that 
\begin{equation}\label{eq:FypyTLclkfdiff}
    F(Y_p,\dots,Y_{T-1}) = \| L_{\clk,X}(f) - L_{\clk}f\|_{\bbH}.
    \end{equation}
    We show that there is a constant $C_F \in (0,\infty)$ such that, for all $T \in \NN$ and 
      \[
      \widetilde{y} = (y_p,\dots, y_{i-1}, \widetilde{z}, y_{i+1},\dots y_{T-1}), \quad \bar{y} = (y_p, \dots, y_{i-1},\bar{z},y_{i+1},\dots, y_{T-1}) \in (\clx^p)^{T-p},
      \]
      we have 
\begin{equation}\label{eq:dedeckerseparatelyboudned}
 |F(\widetilde{y}) - F(\bar{y})| \le  C_F T^{-1}
    \end{equation}
with $C_F \doteq 4  \sqrt{d} \kappa \| f\|_{\infty}$.
The proof of \eqref{eq:dedeckerseparatelyboudned} can be found below. If \eqref{eq:dedeckerseparatelyboudned} holds, then, with further explanations given below,
\begin{align}
   & \PP( F(Y_p,\dots,Y_{T-1}) > \delta) \\
   &= \PP( F(Y_p,\dots,Y_{T-1})  - \EE(F(Y_p,\dots,Y_{T-1})) > \delta - \EE(F(Y_p,\dots,Y_{T-1}))) \nonumber \\
    &\le  \PP( F(Y_p,\dots,Y_{T-1})  - \EE(F(Y_p,\dots,Y_{T-1})) > \delta - C_f T^{-1/2})
    \label{eq:Fdedeckercon1-1}
    \\
    &\le 2 \exp\left( - \frac{ C_{\operatorname{mc}} }{C^2_F} T (\delta - C_f T^{-1/2})^2 \right)
    \le c_1 T^{-c_2},
    \label{eq:Fdedeckercon1}
\end{align}
where \eqref{eq:Fdedeckercon1-1} follows by Lemma \ref{prop:bound_expected_value}. Due to Lemma \ref{prop:geomergodic}, $\{Y_t\}$ is a geometrically ergodic Markov chain, so we can apply Theorem 0.2 of \cite{dedecker2015subGaussian}. This result, together with \eqref{eq:dedeckerseparatelyboudned}, ensures that there is a $C_{\operatorname{mc}} \in (0,\infty)$ such \eqref{eq:Fdedeckercon1} holds. The final inequality then follows for some constants $c_1,c_2 \in (0,\infty)$, since $\delta \geq ({C_F}/{ \sqrt{C_{\operatorname{mc}} }}) \sqrt{{\log(T)}/{T}}$.

Thus, it suffices to prove \eqref{eq:dedeckerseparatelyboudned}. Observe that 
\begin{equation}\label{eq:fxbarminusfybar}
    \begin{split}
    &|F(\widetilde{y}) - F(\bar{y})|
    \\
    &= \left| \left\| \frac{1}{T} \sum\nolimits_{t=p+1}^T f(\widetilde{y}_{t-1}) \clk(\widetilde{y}_{t-1},\cdot) - L_{\clk}f\right\|_{\bbH} - \left\| \frac{1}{T} \sum\nolimits_{t=p+1}^T f(\bar{y}_{t-1}) \clk(\bar{y}_{t-1},\cdot) - L_{\clk}f\right\|_{\bbH}  \right|\\
    &\le  \left\|  \frac{1}{T} \sum\nolimits_{t=p+1}^T f(\widetilde{y}_{t-1}) \clk(\widetilde{y}_{t-1},\cdot) -  \frac{1}{T} \sum\nolimits_{t=p+1}^T f(\bar{y}_{t-1}) \clk(\bar{y}_{t-1},\cdot)  \right\|_{\bbH} \\
    &=  \frac{1}{T} \left\|  f(\widetilde{z})\clk(\widetilde{z},\cdot) -f(\bar{z})\clk(\bar{z},\cdot) \right\|_{\bbH}\\
    &\le \frac{1}{T} \sup\limits_{w, \widetilde{w} \in \clx^p} \| f(w) \calK( w,\cdot) - f(\widetilde{w}) \calK( \widetilde{w},\cdot)\|_{\bbH}
    \le \frac{1}{T} C_F.
    \end{split}
\end{equation}
The final inequality in \eqref{eq:fxbarminusfybar} is verified in the calculations below.

For $w, \widetilde{w} \in \clx^p$ we have
\begin{equation}\label{eq:fxbarkminusfybark}
    \begin{split}
        &\| f(w) \calK( w,\cdot) - f(\widetilde{w}) \calK( \widetilde{w} ,\cdot)\|_{\bbH} \\
        &\le \| f(w) \calK( w,\cdot) - f(w) \calK( \widetilde{w},\cdot)\|_{\bbH} + \| f(w) \calK( \widetilde{w},\cdot) - f(\widetilde{w}) \calK( \widetilde{w},\cdot)\|_{\bbH}\\
        &= \sum\nolimits_{i=1}^{d} \left( \|f_i(w) (K_i(w,\cdot) - K_i(\widetilde{w},\cdot)) \|_{\bbH_i} + \|f_i(w)K_i(\widetilde{w},\cdot) - f_i(\widetilde{w})K_i(\widetilde{w},\cdot)\|_{\bbH_i} \right)\\
        &\le (2 + \sqrt{2}) \sum\nolimits_{i=1}^{d} \kappa_i
  \|f\|_{\infty}
  \le 4  \sqrt{d} \kappa \| f\|_{\infty},
    \end{split}
\end{equation}
where the second inequality follows on observing that, due to Assumption \ref{ass:kernelbound},
\begin{equation*}
    \begin{split}
        \| f_i(w) (K_i(w,\cdot) - K_i(\widetilde{w}, \cdot))\|_{\bbH_i} &\le \|f\|_{\infty} \left(\| K_i(w,\cdot)\|_{\bbH_i} + \| K_i(\widetilde{w},\cdot)\|_{\bbH_i} \right)
        \le 2\|f\|_{\infty}  \kappa_i
    \end{split}
\end{equation*}
  and
\begin{equation*}
    \begin{split}
\|f_i(w)K_i(\widetilde{w},\cdot) - f_i(\widetilde{w})K_i(\widetilde{w},\cdot)\|_{\bbH_i} &\le 2\|f_i\|_{\infty} \| K_i(\widetilde{w},\cdot)\|_{\bbH_i} \le 2 \| f\|_{\infty} \kappa_i.
    \end{split}
\end{equation*}
Combining  \eqref{eq:fxbarminusfybar} and \eqref{eq:fxbarkminusfybark}, we see that \eqref{eq:dedeckerseparatelyboudned} holds with our choice of $C_F$. The result follows. 
\end{proof}

Lemma \ref{pro:extremevalepsilons} establishes a concentration inequality for the maximum of the noise sequence, which is used in Theorem \ref{th:finiteMercer}.
\begin{Lem}\label{pro:extremevalepsilons}
Recall $\eta$ from \eqref{eq:xdef635} and suppose Assumption \ref{ass:subgaussiannoise}. Then, for any $\gamma > \sqrt{2\sigma^2 \log(dT)}$, 
\[
\PP( \| \eta\|_{\infty} > \gamma ) \le 
\exp\left(- \left( \gamma - \sqrt{2\sigma^2 \log(dT)} \right)^2/(2\sigma^2) \right).
\]
\end{Lem}
\begin{proof}
Note that
\begin{align}
\Prob\left( 
\|\eta\|_{\infty} > \gamma
\right)
&=
\Prob\Bigg( 
\max_{t=1,\dots,T} \| \varepsilon_{t} \|_{\infty} - \EE \left(\max_{t=1,\dots,T} \| \varepsilon_{t} \|_{\infty}\right) > \gamma - \EE \left(\max_{t=1,\dots,T} \| \varepsilon_{t} \|_{\infty}\right)
\Bigg)
\nonumber
\\&\leq
\exp\left(- \left( \gamma - \EE \left( \max\limits_{t=1,\dots,T} \| \varepsilon_{t} \|_{\infty}\right)\right)^2/(2\sigma^2) \right)
\label{eq:epsevent1}
\\&\le
\exp\left(- \left( \gamma - \sqrt{2\sigma^2 \log(dT)} \Big)^2\right)/(2\sigma^2) \right),
\label{eq:epsevent2}
\end{align}
where we applied Example 2.29 of \cite{wainwright2019high} with $\sigma^2 = \max_{i=1,\dots,d} \sigma_i^2$ in \eqref{eq:epsevent1} and assuming that $\gamma > \EE (\max_{t=1,\dots,T} \| \varepsilon_t \|_{\infty})$.
In \eqref{eq:epsevent2}, we use that 
\[
\EE ( \max\nolimits_{t=1,\dots,T} \| \varepsilon_t \|_{\infty} ) \leq \sqrt{2\sigma^2 \log(dT)},
\]
which is due to Exercise 2.5.10 in \cite{vershynin2018high}.
\end{proof}

Lemma \ref{prop:bound_expected_value} is used in the proof of Lemma \ref{lem:newproximalvstrueregfunction}.

\begin{Lem} \label{prop:bound_expected_value}
Let $f = (f_1,\dots,f_d)' : \clx^p \to \RR^d$ be a bounded map. For each $T \in \NN$, let  $F_T : (\clx^p)^{T-p} \to \RR$ be given by 
\begin{equation*} 
        F_T(y_p,\dots,y_{T-1}) \doteq \left\| \frac{1}{T} \sum\nolimits_{t=p+1}^{T} f(y_{t-1}) \mathcal{K}(y_{t-1}, \cdot) - L_{\mathcal{K}}(f) \right\|_{\bbH}.
    \end{equation*}
Then, there is a constant $C_f \in (0,\infty)$ such that for all $T \in \NN$, 
\[
\EE ( F_T(Y_p,\dots,Y_{T-1})) \le T^{-1/2} C_f.
\]
\end{Lem}

\begin{proof}

Let $h : \clx^p \to \mathbb{H}$ be given by 
\begin{equation}\label{eq:dfehfclkvecx}
h( x ) =f(x) \mathcal{K}(x, \cdot), \quad  x \in \clx^p,
\end{equation}
and consider the map $H : \clx^p \times \clx^p \to \RR$ given by  
\[
H(x , y) = \langle h(x), h(y) \rangle_{\bmH}, \quad x, y \in \clx^p.
\]
Then, by the Cauchy-Schwarz inequality, and with further explanations given below,
\begin{equation}\label{eq:caphfunboundedinhnormuniformly}
\sup\nolimits_{x,y \in \clx^p} | H(x, y) | 
\le
\sup\nolimits_{x,y \in \clx^p} \| h(x )\|_{\bmH} \| h(y) \|_{\bmH}
\le \widehat{c}_1 \doteq \|f\|_{\infty}^2 \kappa \sqrt{d}.
\end{equation}
The last inequality in \eqref{eq:caphfunboundedinhnormuniformly} follows, since, by Lemma \ref{prop:clkkernelbasicprops},
\begin{equation} \label{eq:hfunboundedinhnormuniformly}
\begin{split} 
    \| h(x)\|_{\bmH}^2 
    &= 
    \langle f(x) \clk(x, \cdot),  f(x) \clk(x, \cdot)  \rangle_{\bmH}
    = 
    \sum\nolimits_{i=1}^{d} \langle f_i(x) K_i(x, \cdot), f_i(x)K_i(x, \cdot) \rangle_{\bmH}\\
    &\le 
    \sum\nolimits_{i=1}^{d} \| f_i\|_{\infty}^2 K_i(x, x)
    \le 
    \| f\|_{\infty}^2\sum\nolimits_{i=1}^{d}  K_i(x, x)
    \le \kappa \| f\|_{\infty}^2 \sqrt{d}.
\end{split}
\end{equation}
Additionally,  Lemma \ref{prop:geomergodic} tells us that
there are $\pi \in \mathcal{P}(\clx^p)$, $\rho \in (0,1)$, and a $\pi$-integrable   function $J : \clx^p \to [0,\infty)$ such that, for each $s,t  \in \NN$, 
\begin{equation}\label{eq:geomusecovcalc1}
        \| P^{|s-t|}(x,\cdot) - \pi(\cdot) \|_{\text{TV}} \leq \rho^{|s-t|} J(x), \quad  x \in \clx^p,
    \end{equation}
    where $P$ is the transition kernel of the process $\{Y_t\}$ defined in \eqref{eq:ynkernel1}. Then, for $s, t \in \NN$,  \eqref{eq:caphfunboundedinhnormuniformly} and  \eqref{eq:geomusecovcalc1} tell us that 
\begin{equation}
\begin{split}
 & \left| \EE_{\pi}( H (Y_{t-1}, Y_{s-1})) 
 -   \int_{\clx^p\times \clx^p}H(x,y) \pi(dy)\pi(dx) \right| \\
 &= \left|\int_{\clx^p \times \clx^p} H(x, y) P^{|s-t|} (x,dy) \pi(dx) -   \int_{\clx^p\times \clx^p}H(x,y) \pi(dy)\pi(dx) \right|\\
 &= \left| \int_{\clx^p \times \clx^p} H(x,y) \left( P^{|s-t|}(x,dy) - \pi(dy)\right)\pi(dx)\right|\\
 &\le \widehat{c}_1 \int_{\clx^p \times \clx^p} \left| P^{|s-t|}(x, dy) - \pi(dy)\right|  \pi(dx)\\
 &\le 2 \widehat{c}_1 \int_{\clx^p} \rho^{|s-t|}J(x) \pi(dx)
 = \widehat{c}_2 \rho^{|s-t|}, \label{eq:rateconvehnorm}
\end{split}
\end{equation}
where $\widehat{c}_2 \in [0, \infty)$ is defined as $\widehat{c}_2 = 2\widehat{c}_1 \int_{\clx^p} J(x) \pi(dx)$.
If  $\Pi_1,\Pi_2 \stackrel{iid}{\sim} \pi$, then, using Lemma \ref{prob:ind_orthogonal}, we see that 
\begin{equation}
\begin{split}
&
\int_{\clx^p \times \clx^p} H(x, y ) \pi(dy) \pi(dx) 
= \EE[ \langle h(\Pi_1), h(\Pi_2) \rangle_{\mathbb{H}}]\\
&=  \langle \EE_{\bbH} (h(\Pi_1) ), \EE_{\bbH}( h(\Pi_2))\rangle_{\bbH}
= \langle \pi(h) , \pi(h) \rangle_{\bbH} 
= \| \pi(h)\|_{\bbH}^2 ,\label{eq:hnormpiypix1395}
\end{split}
\end{equation}
where
\[
\pi(h) (y) = \int_{\clx^p} f(x) \clk(x, y) \pi(dx).
\]
Above we have used the fact that $\pi(h) \in \bbH$; this can be shown using Lemma \ref{prob:ind_orthogonal}  and the fact that $h(y) \in \bbH$ for each $y \in \clx^p$.
Together, \eqref{eq:rateconvehnorm} and \eqref{eq:hnormpiypix1395} imply that 
\begin{align}\label{eq:epiyhminuspuyhnormh1}
   |  \EE_{\pi}( H (Y_{t-1}, Y_{s-1})) - \| \pi(h)\|^2_{\bbH} | \le \widehat{c}_2 \rho^{|s-t|}. 
\end{align}
Thus, using the fact that $\EE_{\pi}(h(Y_{t})) = \pi(h)$ for all $t$, we have  
\begin{align}
    &\Cov_{\mathbb{H}}( h(Y_t) - \pi(h), h(Y_s) - \pi(h) ) 
    = \EE_{\pi}[ \langle h(Y_t) - \pi(h), h(Y_s) - \pi(h) \rangle_{\mathbb{H}} ]\nonumber\\
    &= \EE_{\pi}[ \langle h(Y_t), h(Y_s) \rangle_{\mathbb{H}}] -   \langle \pi(h), \pi(h) \rangle_{\mathbb{H}} -  \langle \pi(h), \pi(h) \rangle_{\mathbb{H}} + \langle \pi(h), \pi(h) \rangle_{\mathbb{H}}\nonumber\\
    &= \EE_{\pi}[ \langle h(Y_t) , h(Y_s) \rangle_{\mathbb{H}}] - \| \pi(h)\|_{\mathbb{H}}^2
    = \EE_{\pi} (H(Y_t,Y_s)) - \|\pi(h)\|_{\bmH}^2,
    \label{eq:epiyhminuspuyhnormh2}
\end{align}
so \eqref{eq:epiyhminuspuyhnormh1} and \eqref{eq:epiyhminuspuyhnormh2} imply that 
\begin{equation}\label{eq:covariancebound1457}
| \Cov_{\mathbb{H}}( h(Y_t), h(Y_s) ) | = | \Cov_{\mathbb{H}}( h(Y_t) - \pi(h), h(Y_s) - \pi(h) ) | \le \widehat{c}_2 \rho^{|s-t|}.
\end{equation}
We also note that, due to \eqref{eq:hfunboundedinhnormuniformly}, 
\begin{equation}\label{eq:varaincebound1457}
\begin{split}
\sup\nolimits_{t\in \NN}\Var_{\bbH}(h(Y_{t-1})) \le \widehat{c}_3 \doteq  \sup\nolimits_{x\in\clx^p}\|h(x)\|_{\bbH}^2 < \infty.
\end{split}
\end{equation}
Define $\mathcal{I}_{p:T} = \{(i,j) : i, j \in \{p+1,\dots,T\} \text{ and } i \ne j\}$,
then,
\begin{align}
    &
    \left( \EE_{\pi} \left(F(Y_p,\dots, Y_{T-1}) \right)\right)^2
\leq
\EE_{\pi} \left(F(Y_p, \dots, Y_{T-1}) \right)^2
\label{al:eq-bound-expectation-A1}
\\&=
    \EE_{\pi} \left[\left\| \frac{1}{T} \sum\nolimits_{t=p+1}^{T} f(Y_{t-1}) \mathcal{K}(Y_{t-1}, \cdot) - L_{\mathcal{K}}(f) \right\|^2_{\bbH} \right]
\\&= \frac{1}{T^2}
    \EE_{\pi} \left[ \left\| \sum\nolimits_{t=p+1}^{T} \left( h(Y_{t-1}) - \pi(h) \right)\right\|^2_{\bbH} \right]
    \label{al:eq-bound-expectation-A2}\\
    &= \frac{1}{T^2} \sum\nolimits_{t=p+1}^{T} \Var_{\mathbb{H}} (h (Y_{t-1}))  + \frac{1}{T^2}\sum\nolimits_{(s,t) \in \mathcal{I}_{p:T}} \Cov_{\mathbb{H}}(h(Y_{t-1}), h(Y_{s-1}))\\
    &\le \frac{\widehat{c}_3}{T} + \frac{\widehat{c}_2}{T^2} \sum\nolimits_{s,t=1}^{T} \rho^{|s-t|}
    \label{al:eq-bound-expectation-A3}\\
    &\le \frac{1}{T} \left(  \widehat{c}_3 +  \widehat{c}_2\left(1 + \frac{2\rho}{1-\rho}\right)\right) \doteq \frac{1}{T} C_f^2,
\end{align}
where \eqref{al:eq-bound-expectation-A1} uses Jensen's inequality, \eqref{al:eq-bound-expectation-A2} uses \eqref{eq:dfehfclkvecx} and the fact that $\pi(h) =  L_{\mathcal{K}}(f)$, and \eqref{al:eq-bound-expectation-A3} uses \eqref{eq:covariancebound1457} and \eqref{eq:varaincebound1457}. 
\end{proof}

The next lemma is also used in the proof of Theorem \ref{th:finiteMercer}.
\begin{Lem}\label{pro:separetelybounded}
Let $F_{i,M} : (\clx^p)^{T-p+1} \to [0,\infty)$ defined by
\begin{equation*}
\begin{split}
   & F_{i,M}(y_p,\dots,y_T)\\
   &= \max_{k=1,\dots,M} \max_{j =1, \dots, N(k)} \frac{1}{T}
    \left| \sum_{t=p+1}^{T}  \left((y_{t,1} - f(y_{t-1}))_i \beta_{i,j,k}^{-1} \lambda^{\frac{1}{2}}_{i,k} \phi_{i,j,k} (y_{t-1})
    \bm{1}_{\{ | (y_{t,1} - f(y_{t-1}))_i | \leq \gamma \}} \right)
    \right|,
\end{split}
\end{equation*}
    where $\beta_{i,j,k}$ is defined as in \eqref{eq:assumptionbetasdef}. Then, $F_{i,M}$ satisfies a separately bounded condition in the sense of Theorem 0.2. of \cite{dedecker2015subGaussian}. Namely, for all $T \in \NN$, all $M \le T$ and 
      \[
      \widetilde{y} = (y_p,\dots, y_{j-1}, \widetilde{z}, y_{j+1},\dots y_{T}), \quad \bar{y} = (y_p, \dots, y_{j-1},\bar{z},y_{j+1},\dots, y_T) \in (\mathcal{X}^p)^{T-p+1},
      \]
      we have 
    \begin{equation}\label{eq:dedeckerseparatelyboudned2}
    |F_{i,M}(\widetilde{y}) - F_{i,M}(\bar{y})| \le  \frac{2 \gamma}{T}.
    \end{equation}
\end{Lem} 
\begin{proof} 
Then, 
\begin{align}
    &
    |F_{i,M}(\widetilde{y}) - F_{i,M}(\bar{y})| \nonumber
    \\
    &\le \max_{k=1,\dots,M} \max_{j =1, \dots, N(k)} \frac{1}{T}
    \Big| \sum_{t=p+1}^{T} ( (\widetilde{y}_{t,1} - f(\widetilde{y}_{t-1}))_i \beta_{i,j,k}^{-1} \lambda^{\frac{1}{2}}_{i,k} \phi_{i,j,k} (\widetilde{y}_{t-1}) 
    \bm{1}_{\{ |(\widetilde{y}_{t,1} - f(\widetilde{y}_{t-1}))_i |  \leq \gamma \}}
    \nonumber
    \\&\hspace{3cm}-
    (\bar{y}_{t,1} - f(\bar{y}_{t-1}))_i \beta_{i,j,k}^{-1} \lambda^{\frac{1}{2}}_{i,k} \phi_{i,j,k} (\bar{y}_{t-1}) 
    \bm{1}_{\{ | (\bar{y}_{t,1} - f(\bar{y}_{t-1}))_i | \leq \gamma \}}
    \Big| \nonumber\\
    &=   \max_{k=1,\dots,M} \max_{j =1, \dots, N(k)} \frac{1}{T}
    \Big| \sum_{t=j}^{j+1} ( (\widetilde{y}_{t,1} - f(\widetilde{y}_{t-1}))_i \beta_{i,j,k}^{-1} \lambda^{\frac{1}{2}}_{i,k} \phi_{i,j,k} (\widetilde{y}_{t-1}) 
    \bm{1}_{\{ | (\widetilde{y}_{t,1} - f(\widetilde{y}_{t-1}))_i |  \leq \gamma \}}
    \nonumber
    \\&\hspace{3cm}-
    (\bar{y}_{t,1} - f(\bar{y}_{t-1}))_i \beta_{i,j,k}^{-1} \lambda^{\frac{1}{2}}_{i,k} \phi_{i,j,k} (\bar{y}_{t-1}) 
    \bm{1}_{\{ | (\bar{y}_{t,1} - f(\bar{y}_{t-1}))_i | \leq \gamma \}}
    \Big| 
    \le \frac{2\gamma}{T},
    \nonumber
\end{align}
where the final inequality uses  that, by definition of $\beta_{i,j,k}$ in \eqref{eq:assumptionbetasdef}, $\beta_{i,j,k}^{-1} \lambda_{i,k}^{1/2} \| \phi_{i,j,k}\|_{\infty}\le 1$,
and
\[
|(y_{t,1} - f({y}_{t-1}))_i \bm{1}_{\{ | (y_{t,1} - f(y_{t-1}))_i | \le \gamma \}}| \le \gamma, \quad y_t,y_{t-1} \in \clx^p.
\]
\end{proof}

\section{Results for Gaussian kernels and their proofs} \label{se:ProofsGaussianKernel}

Recall the definition of a Gaussian kernel \eqref{eq:gaussiankerneldef1}. In this section, we show that the Gaussian kernel satisfies Assumptions \ref{ass:kernelbound}--\ref{ass:momentboundtail}.

Lemma \ref{le:Gaussian-Mercer-rep} provides an eigenfunction expansion of the Gaussian kernel and states that it satisfies Assumption \ref{ass:kernelbound}.

\begin{Lem} \label{le:Gaussian-Mercer-rep}
Recall the set $\cln(k)$ from \eqref{eq:def-mathcalN(k)}.
Enumerate $\mathcal{N}(k)$ as $
\mathcal{N}(k) = \{ \bm{n}_{k,1},\dots,\bm{n}_{k,N(k)}\}$,
where, for each $j =1,\dots, N(k)$,
\begin{equation}\label{eq:boldnkjvectorized}
\bm{n}_{k,j} = (n_{k,j, 1,1},\dots, n_{k,j,1,p}, \dots, n_{k,j,d,1},\dots, n_{k,j,d,p}).
\end{equation}
For $x \in \clx^p$ and $\bm{n}_{k,j} \in \cln(k)$, let
$x^{\bm{n}_{k,j}} \doteq \prod\nolimits_{i=1}^{d} \prod\nolimits_{r=1}^{p}  x_{i,r}^{ n_{k,j,i,r}}$.
The Gaussian kernel in \eqref{eq:gaussiankerneldef1} satisfies Assumptions \ref{ass:kernelbound} and  \ref{ass:separable_kernel_multi} with $\phi_{i,j,k} : \clx^p \to \RR$, $i = 1,\dots, d$, $k \in \NN$, $j = 1,\dots,N(k)$, given by 
\begin{equation}\label{eq:phigaussiandist}
 \phi_{i,j,k}(x)
    =  \exp\left( -\frac{ \|x\|^2}{\tau^2}\right) \binom{k}{ \bm{n}_{k,j}}^{1/2} x^{\bm{n}_{k,j}} , \quad x \in \clx^p,
\end{equation}
and 
\begin{equation}\label{eq:eigenvaluedef1}
    \lambda_{i,k} = \frac{1}{k!} \left( \frac{2}{\tau^2}\right)^k.
\end{equation}
\end{Lem}

\begin{proof}
The Gaussian kernel satisfies Assumption \ref{ass:kernelbound} with each $\kappa_i = 1$. Finally, the fact that the Gaussian kernel satisfies Assumption \ref{ass:separable_kernel_multi} follows from the Taylor expansion of the map $u \mapsto \exp(u)$.
\end{proof}

\begin{Rem}
    For the Gaussian kernel, the convergence in \eqref{eq:mercerexpansion} (with $M = \infty$) holds pointwise but not uniformly over $\clx^p \times \clx^p$. 
\end{Rem}

The following lemma establishes an upper bound for exponentially-weighted second moments of Gaussian random variables, and is used in the proof of Lemma \ref{pro:onedimensionalgaussiankernelmoments} to compute second moments of the eigenfunctions \eqref{eq:phigaussiandist}. 

\begin{Lem}\label{prop:normal_moment_zalpha}
Suppose that $X_{\alpha} \sim \cln( \alpha, \sigma^2)$, for $|\alpha| \le L$. Then, there is a constant $C_L \in (0,\infty)$ such that for each $k \in \NN$,
\[
\sup\limits_{|\alpha| \le L} \EE \left[ X^{2k}_{\alpha} \exp\left( - 2\frac{X^2_{\alpha}}{\tau^2}\right)\right] \le C_L \left(\frac{\sigma^2 \tau^2}{4\sigma^2+\tau^2}\right)^k 2^k (k+1)!.
\]
\end{Lem}

\begin{proof}
Observe that 
    \begin{align}
    &
        \EE \left(X^{2k}_{\alpha} \exp\left(-2\left(\frac{X_{\alpha}}{\tau}\right)^2 \right) \right) \nonumber \\
        &=
        \int_{-\infty}^{\infty} x^{2k} \exp\left(-2\left(\frac{x}{\tau}\right)^2\right) 
        \exp\left( -\frac{(x-\alpha)^2}{2\sigma^2} \right) dx
        \nonumber \\&=
       \exp\left( - \frac{\alpha^2}{2} + \frac{\alpha^2 \tau^4}{4\sigma^2 +\tau^2} \right)  \int_{-\infty}^{\infty} x^{2k}  
        \exp\left( -\frac{ -(4\sigma^2+\tau^2)x^2 + 2\alpha x \tau^2 -\alpha^2 \tau^4 /(4\sigma^2+\tau^2))}{2\sigma^2 \tau^2} \right) dx
\nonumber        \\&=
         \left( \frac{2 \pi \tau^2}{4 \sigma^2 + \tau^2}\right)^{1/2}\exp\left( \alpha^2 \left( \frac{\tau^4}{4 \sigma^2 + \tau^2} - \frac{1}{2\sigma^2} \right) \right)  \EE (Y^{2k}_{\alpha}), \label{eq:x2kalphamomentsexp}
    \end{align}
where $Y_{\alpha} \sim \mathcal{N}(\mu_{\alpha}, \sigma^2_Y)$, $\mu_{\alpha} = \frac{\tau^2 \alpha}{4\sigma^2+\tau^2}$, and $\sigma_Y^2 = \frac{\sigma^2 \tau^2}{4\sigma^2+\tau^2}$. Let ${}_1F_1$ denote the confluent hypergeometric function of the first kind (see, e.g., \cite{LUKE197241}, \cite{JOSHI1996251}). Then, by the moment 
formulas for Gaussian random variables in \cite{winkelbauer2012moments}, we have
\begin{align}
    \EE Y^{2k}_{\alpha}
    &=  \pi^{-1/2} \sigma_Y^{2k} 2^{k} \Gamma\left(\frac{1 + 2k}{2}\right) {}_1F_1\left(-k,\frac{1}{2}, -\frac{1}{2} \left(\frac{\mu_{\alpha}}{\sigma_Y}\right)^2 \right)
    \nonumber
    \\&\le
   \pi^{-1/2} \sigma_Y^{2k} 2^{k} \Gamma\left(\frac{1 + 2k}{2}\right) \exp\left(-\frac{1}{2} \left(\frac{\mu_{\alpha}}{\sigma_Y}\right)^2\right)
    {}_1F_1\left(\frac{1}{2} + k,\frac{1}{2}, \frac{1}{2} \left(\frac{\mu_{\alpha}}{\sigma_Y}\right)^2 \right)
    \label{eq:1F1_al1}
    \\&\leq \pi^{-1/2} 
    \sigma_Y^{2k} 2^{k} \Gamma\left(\frac{1 + 2k}{2}\right) \exp\left(-\frac{1}{2} \left(\frac{\mu_{\alpha}}{\sigma_Y}\right)^2\right)
    \left( 1 +  (1 +2k) \left(\frac{\mu_{\alpha}}{\sigma_Y}\right)^2 \right)
    \label{eq:1F1_al2} \\
    &\le \sigma_Y^{2k} 2^{k} k!
    \left( 1 +  (1 +2k) \left(\frac{\mu_{\alpha}}{\sigma_Y}\right)^2 \right),
    \label{eq:1F1_al2+1}
\end{align}
where \eqref{eq:1F1_al1} follows from (1.18) in \cite{LUKE197241} and \eqref{eq:1F1_al2} is due to (4.3) in \cite{JOSHI1996251}.
Combining \eqref{eq:x2kalphamomentsexp} and \eqref{eq:1F1_al2+1}, we see that the result holds with 
\begin{equation}\label{eq:CLdef}
C_{L} \doteq 2 \left( \frac{2 \pi \tau^2}{4 \sigma^2 + \tau^2}\right)^{1/2} \left(\sup\limits_{|\alpha| \le L}   \left(\exp\left( \alpha^2 \left( \frac{\tau^4}{4 \sigma^2 + \tau^2} - \frac{1}{2\sigma^2} \right) \right)\right)\right)\left( \max\left\{ \sup\limits_{|\alpha|\le L} \left(\frac{\mu_{\alpha}}{\sigma_Y}\right)^2, 1 \right\} \right).
\end{equation}
\end{proof}

\begin{Lem}\label{pro:onedimensionalgaussiankernelmoments}
For $k \in \NN$ and $j = 1, \dots, N(k)$, recall from \eqref{eq:phigaussiandist} the eigenfunctions of the Gaussian kernel.  
Then, there is some $b \in (0,\infty)$ such that for all $k \in \NN$, $j = 1, \dots, N(k)$, $i = 1,\dots,d$, and $t \in \NN$, $t \geq p+1$,
\[
\EE( (\phi_{i,j,k}(X_{(t-p):(t-1)}))^2)  \le b \left( \frac{2\sigma^2 \tau^2}{4 \sigma^2 + \tau^2}\right)^{k} k! (k+1)^{pd}.
\]
\end{Lem}

\begin{proof}
Recall the quantity $\bm{n}_{k,j}$ introduced in Lemma \ref{le:Gaussian-Mercer-rep}. Since $k$ and $j$ are fixed, we, with a slide abuse of notation, write $\bm{n}$ in place of $ \bm{n}_{k,j}$, where $\bm{n} \doteq ( n_{1,1},\dots, n_{1,p}, \dots, n_{d,1}, \dots, n_{d,p})$. Let
\begin{equation}
Z_{\alpha,i,r } \sim \cln(\alpha_i, \sigma^2_i),  \quad \alpha = (\alpha_1,\dots,\alpha_d)' \in \RR^{d}, \; i = 1,\dots, d, \; r = 1,\dots, p,
\label{eq:zalphadef1}
\end{equation}
and note that Lemma \ref{prop:normal_moment_zalpha} ensures that there is some $b \in (0,\infty)$ such that  
\begin{equation}\label{eq:supalphaleginfzmoment1}
\sup\limits_{\| \alpha\| \le  \|g\|_{\infty}} \EE \left( \exp\left( - \frac{2 Z^{2}_{\alpha, i,r}}{\tau^2}\right) Z^{2 n_{i,r}}_{\alpha,i,r}\right) \le C \left( \frac{\sigma^2_i \tau^2}{4 \sigma^2_i + \tau^2}\right)^{n_{i,r}} 2^{n_{i,r}} (n_{i,r}+1)!.
\end{equation}
With further explanations given below, the second moments of the eigenfunctions can then be bounded as
    \begin{align}
    \EE( (\phi_{i,j,k}(X_{(t-p):(t-1)}))^2) 
    &= 
  \binom{k}{\bm{n}}  \EE \left[\left( \exp\left( - \frac{ \| X_{(t-p):(t-1)} \|^2_2 }{\tau^2} \right) \prod\limits_{r=1}^{p} \prod_{i=1}^d X_{i,t-r}^{n_{i,r}}\right)^2\right]
    \nonumber
    \\&= 
   \binom{k}{\bm{n}}  \EE \left(  \prod_{r=1}^p \prod\limits_{i=1}^{d}  \exp\left( - \frac{ 2 X_{i,t-r}^2 }{\tau^2} \right)   X_{i,t-r}^{2 n_{i,r}}\right)
    \nonumber
    \\&\leq
   \binom{k}{\bm{n}}    \prod\limits_{r=1}^p \prod_{i=1}^d \sup_{\| \alpha \| \leq \| g \|_{\infty} } \EE \left( \exp\left( - \frac{ 2 Z_{\alpha,i,r}^2 }{\tau^2} \right)  Z_{\alpha,i,r}^{2 n_{i,r}} \right)
    \label{al:boundmomentsphial3}
    \\&\le
b  \binom{k}{\bm{n}}   \prod\limits_{r=1}^{p} \prod_{i=1}^d \left( \frac{\sigma^2_i \tau^2}{4 \sigma^2_i + \tau^2}\right)^{n_{i,r}} 2^{n_{i,r}} (n_{i,r}+1)!
\label{al:boundmomentsphial4}
    \\&\leq b \binom{k}{\bm{n}} 
    \left( \frac{\sigma^2 \tau^2}{4 \sigma^2 + \tau^2}\right)^{k} 2^{k} \prod\limits_{r=1}^{p} \prod_{i=1}^d  (n_{i,r}+1)!,
    \label{al:boundmomentsphial5}
\end{align}
where \eqref{al:boundmomentsphial3} is proved below and \eqref{al:boundmomentsphial4} follows from \eqref{eq:supalphaleginfzmoment1}.
For \eqref{al:boundmomentsphial3}, we recursively apply the following estimate:
\begin{align}
    &
    \EE \left(  \prod_{r=1}^p \prod\limits_{i=1}^{d}  \exp\left( - \frac{ 2 X_{i,t-r}^2 }{\tau^2} \right)   X_{i,t-r}^{2 n_{i,r}}\right)
    \nonumber
    \\&=
    \EE \left( \prod_{r=2}^{p} \prod_{{i}=1}^d \exp\left( - \frac{ 2 X_{i,t-r}^2 }{\tau^2} \right)  X_{i,t-r}^{2 n_{i,r}} 
    \EE \left[ 
    \prod_{{i}=1}^d \exp\left( - \frac{ 2 X_{i,t-1}^2 }{\tau^2} \right)  X_{i,t-1}^{2 n_{i,1}}
    \mid 
    X_{(t-p-1):(t-2)} \right] \right)
    \label{al:boundmomentsphial1}
    \\&=
    \EE \left( \prod_{r=2}^{p} \prod_{{i}=1}^d \exp\left( - \frac{ 2 X_{i,t-r}^2 }{\tau^2} \right)  X_{i,t-r}^{2 n_{i,r}} 
    \prod_{{i}=1}^d 
    \EE \left[ 
    \exp\left( - \frac{ 2 X_{i,t-1}^2 }{\tau^2} \right)  X_{i,t-1}^{2 n_{i,1}}
    \mid 
    X_{(t-p-1):(t-2)} \right] \right)
    \label{al:boundmomentsphial2}
    \\&\leq
    \EE \left( \prod_{r=2}^{p} \prod_{{i}=1}^d \exp\left( - \frac{ 2 X_{i,t-r}^2 }{\tau^2} \right)  X_{i,t-r}^{2 n_{i,r}} \right)
    \prod_{{i}=1}^d
    \sup_{\| \alpha \| \leq \| g \|_{\infty} } \EE \left( \exp\left( - \frac{ 2 Z_{\alpha,i,t-1}^2 }{\tau^2} \right)  Z_{\alpha,i,t-1}^{2 n_{i,1}} \right),
    \label{al:boundmomentsphial2.1}
\end{align}
where the first identity uses the law of total expectation and the second identity  uses that $\Sigma = \diag(\sigma_1^2,\dots,\sigma_d^2)$ is a diagonal matrix. For the inequality in the final line, note that 
$X_{t-1} \mid X_{(t-p-1):(t-2)} \sim \mathcal{N}(g(X_{t-2},\dots,X_{t-2-p}), \Sigma)$. To further bound \eqref{al:boundmomentsphial5}, note that $n_{i,r} \le k$, so  
\begin{equation}\label{eq:prodsofnirplus1}
\begin{split}
  \binom{k}{\bm{n}} \prod\limits_{r=1}^{p} \prod_{i=1}^d  (n_{i,r}+1)! &= \binom{k}{\bm{n}} \prod\limits_{r=1}^{p} \prod_{i=1}^d  (n_{i,r}+1) n_{i,r}!= k! \prod\limits_{r=1}^{p} \prod_{i=1}^d  (n_{i,r}+1)\le k! (k+1)^{pd}
    \end{split}
\end{equation}
Combining \eqref{al:boundmomentsphial5} and \eqref{eq:prodsofnirplus1}, we see that
\[
    \EE( (\phi_{i,j,k}(X_{(t-p):(t-1)}))^2)  \le b
    \left( \frac{2 \sigma^2 \tau^2}{4 \sigma^2 + \tau^2}\right)^{k}  (k+1)^{pd} k!,
\]
as claimed. 
\end{proof}

\begin{Lem}\label{lem:gaussiankernelassumptiontailmoment}
There are $b_3 \in (0,\infty)$ and 
\begin{equation}\label{eq:rho0defgaussian}
\rho_0 \in \left( \frac{4\sigma^2}{4\sigma^2 + \tau^2},  1 \right),\end{equation}
 such that Assumption \ref{ass:momentboundtail} holds for the Gaussian kernel.  
\end{Lem}

\begin{proof}
Letting $b$ be as in the statement of Lemma \ref{pro:onedimensionalgaussiankernelmoments} and fixing $\rho_0$ as in \eqref{eq:rho0defgaussian}, we can find $k_0 \in \NN$ such that 
\begin{equation}\label{eq:geometricboundrho0}
 \left( \frac{ 4\sigma^2 \tau^2}{4\sigma^2 + \tau^2}\right)^k \frac{\left((k+1)(k + dp - 1)\right)^{dp}}{(dp-1)!} \le \rho_0^k, \quad k \ge k_0.
\end{equation}
Then, for each $k \ge k_0$, 
\begin{equation*}
    \begin{split}
         \lambda_{i,k} \sum\nolimits_{j=1}^{N(k)}  \EE((\phi_{i,j,k}(X_{(t-p):(t-1)}))^2) &\le b \frac{1}{k!} \left( \frac{2}{\tau^2}\right)^k \sum\nolimits_{j=1}^{N(k)} \left( \frac{ 2\sigma^2 \tau^2}{4\sigma^2 + \tau^2}\right)^k k!(k+1)^{pd}\\
         &=  b \sum\nolimits_{j=1}^{N(k)} \left( \frac{ 4\sigma^2 }{4\sigma^2 + \tau^2}\right)^k (k+1)^{pd}\\
         &\le b \left( \frac{ 4\sigma^2 }{4\sigma^2 + \tau^2}\right)^k \frac{\left((k+1)(k + dp - 1)\right)^{dp}}{(dp-1)!}
         \le b \rho_0^k,
    \end{split}
\end{equation*}
where the first inequality uses \eqref{eq:eigenvaluedef1} and Lemma \ref{pro:onedimensionalgaussiankernelmoments}, the second inequality uses the fact that
\[
N(k) = \binom{k + dp - 1}{dp -1 } \le \frac{(k + dp - 1)^{dp}}{(dp-1)!},
\]
and the third inequality uses \eqref{eq:geometricboundrho0}. It follows that there is some $b_3 \in (0,\infty)$ such that \eqref{eq:momentboundtaileqn1} holds for all $k \in \NN$.
\end{proof}

Lemma \ref{pro:bound_multinomial} establishes an upper bound for a quantity that arises when computing (suitably rescaled) suprema of the Gaussian kernel's eigenfunctions, and is used in the proof of Lemma \ref{pro:sumbetasgaussian}.

\begin{Lem} \label{pro:bound_multinomial}
Recall $N(k)$ from \eqref{eq:def-N(k)}. Then, for each $k \in \NN$, 
    \begin{equation} \label{eq:bound_multinomial}
        \sum\nolimits_{j=1}^{N(k)}    \prod\nolimits_{\{(i ,r) \in \{1,\dots,d\} \times \{1,\dots, p\} \;  \mid  \;   n_{i,r} > 0\} }   n_{k,j,i,r}^{-1/2}
        \leq \pi^{\frac{dp}{2}} k^{\frac{dp}{2} -1}. 
    \end{equation}
\end{Lem}

\begin{proof}
For notational convenience, we assume that $p = 1$, we enumerate the elements of $\cln(k)$ in \eqref{eq:def-mathcalN(k)} by $\{\bm{n}_1,\dots,\bm{n}_{N(k)}  \}$, and we write $\bm{n}_j = (n_{j,1},\dots,n_{j,d}), \; \; j = 1,\dots,N(k)$.     

We employ an induction principle across $d$. As base case, we show \eqref{eq:bound_multinomial} for $d=2$ (it holds trivially when $d = 1$). First, note that the map $u \mapsto (k-u)^{-1/2}u^{-1/2}$ is decreasing on the interval $[0,k/2]$ and is increasing on the interval $[k/2,k]$, so 
    \begin{align*}
    \sum\limits_{j=1}^{N(k)}   \prod\limits_{i \in \{ 1,2 \; | \;  n_{j,i} > 0\} }    n_{j,i}^{-1/2}
    &=
        \sum_{j=1}^{k-1} \frac{1}{\sqrt{(k- j) j}}
        \sum_{j=1}^{\lfloor k/2 \rfloor} \frac{1}{\sqrt{(k- j) j}} + 
        \sum_{j=\lfloor k/2 \rfloor+ 1}^{k-1} \frac{1}{\sqrt{(k- j) j}}
        \\&
        \leq
        \int_{0}^{\lfloor k/2 \rfloor} \frac{1}{\sqrt{(k- x) x}} dx + 
        \int_{\lfloor k/2 \rfloor+ 1}^{k} \frac{1}{\sqrt{(k- x) x}} dx
        \\&
        \leq
        \int_{0}^{k} \frac{1}{\sqrt{(k- x) x}} dx
        = \pi,
    \end{align*}
    which shows that \eqref{eq:bound_multinomial} holds for $d =2$.    Now, suppose that \eqref{eq:bound_multinomial} holds for some $d \ge 2$. Using the inductive hypothesis, we see that 
\begin{equation}\label{eq:inductionsecondpartproduct}
    \begin{split}
        \sum\limits_{n_1 + \cdots + n_{d+1} = k} \prod\limits_{i=1}^{d+1} n_i^{-1/2} &=\sum\limits_{n_{d+1}=1}^{k-d} \sum\limits_{n_1 + \cdots + n_{d} = k-n_{d+1}} n_{d+1}^{-1/2} \prod\limits_{i=1}^{d} n_i^{-1/2}  \\
        &= \sum\limits_{n_{d+1} = 1}^{k-d} n_{d+1}^{-1/2} \left( \sum\limits_{n_1 + \cdots + n_{d} = k-n_{d+1}} \prod\limits_{i=1}^{d} n_i^{-1/2}  \right)\\
        &\le \pi^{d/2} \sum\limits_{n_{d+1} = 1}^{k-d} n_{d+1}^{-1/2}  (k- n_{d+1})^{d/2-1}.\\
    \end{split}
\end{equation}
Using the fact that, since $d \ge 2$, the map $u \mapsto u^{-1/2} (k - u)^{d/2 -1}$ is  decreasing  on $[0,k]$, we have 
\begin{equation}\label{eq:sumndplus1kminusnd}
    \begin{split}
\sum\limits_{n_{d+1} = 1}^{k-d} n_{d+1}^{-1/2}  (k- n_{d+1})^{d/2-1} &\le \int_0^{k} x^{-1/2} (k-x)^{d/2 - 1} dx\\
&= k^{d/2 - 1/2} \int_0^1 u^{-1/2}(1-u)^{d/2-1} du\\
&= k^{\frac{d+1}{2} -1 }\text{B}(1/2, d/2),\\
\end{split}
\end{equation}
where the third line uses an integral representation of the beta function. Noting that $\Gamma$ is monotonically increasing, we obtain
\begin{equation*}
\text{B}(1/2, d/2) = \frac{ \Gamma(1/2) \Gamma(d/2)}{\Gamma(d/2 + 1/2)} \le \Gamma(1/2) =  \sqrt{\pi}.
\end{equation*}
The result follows on combining the estimate in the previous  display with \eqref{eq:inductionsecondpartproduct} and \eqref{eq:sumndplus1kminusnd}.\phantom{aaa}
\end{proof}

Lemma \ref{pro:sumbetasgaussian} provides upper bounds for the suprema of  (suitably normalized) eigenfunctions of the Gaussian kernel.

\begin{Lem}\label{pro:sumbetasgaussian}
The Gaussian kernel satisfies Assumption \ref{ass:eigenfunctiongrowth} with 
\begin{equation}\label{eq:c1defgaussiankernelassumption}
b_1 \doteq \pi^{\frac{dp}{2}}, \quad 
b_2 \doteq \frac{dp}{2} - 1.
\end{equation}
\end{Lem}

\begin{proof}
As in the proof of Lemma \ref{pro:onedimensionalgaussiankernelmoments}, since $i$ and $k$ are fixed, we write $\bm{n}$ in place of $ \bm{n}_{k,j}$, where $\bm{n} \doteq ( n_{1,1},\dots, n_{1,p}, \dots, n_{d,1}, \dots, n_{d,p})$,
in place of the notation used in \eqref{eq:boldnkjvectorized}.
We first observe that $\phi_{i,j,k}^2$ is maximized at 
\[
\widetilde{x}(j) = (\widetilde{x}_{1,1}(j), \dots, \widetilde{x}_{1,p}(j), \dots, \widetilde{x}_{d,1}(j), \widetilde{x}_{d,p}(j)),
\]
where $
\widetilde{x}_{i,r}(j) = \left( \frac{ \tau^2 n_{i,r}}{2}\right)^{1/2}$. Using Stirling's approximation, we see that 
\begin{align*}
    &\sum\limits_{j=1}^{N(k)} \beta^2_{i,j,k} 
    = 
    \sum\limits_{j = 1}^{N(k)} \sup\limits_{x \in \clx^p } | \lambda_{i,k} \phi_{i,j,k}^2(x) |
    =
     \sum\limits_{j=1}^{N(k)} \frac{1}{k!} \left( \frac{2}{\tau^2}\right)^k\sup\limits_{x\in \clx^p}   \left(  \exp\left( -\frac{2 \|x\|^2}{\tau^2}\right) \binom{k}{\bm{n}}  x^{2\bm{n}}\right)\\
    &= \sum\limits_{j=1}^{N(k)} \frac{1}{k!} \left( \frac{2}{\tau^2}\right)^k  \exp\left( - \frac{2}{\tau^2} \sum\limits_{i=1}^{d}\sum\limits_{r=1}^{p} (\widetilde{x}_{i,r}(j))^2\right) \frac{k!}{\bm{n}!} \prod\limits_{i=1}^{d}\prod\limits_{r=1}^{p} \left( \left( \frac{ \tau^2 n_{i,r}}{2}\right)^{1/2}\right)^{2 n_{i,r}}\\
&= \sum\limits_{j=1}^{N(k)}  \exp\left(- k\right) \prod\limits_{i=1}^{d}\prod\limits_{r=1}^{p} \frac{ n_{i,r}^{n_{i,r}}}{n_{i,r}!}
    =
    \sum\limits_{j=1}^{N(k)}  \exp\left(- k\right) \prod\limits_{\{(i ,r) \in \{1,\dots,d\} \times \{1,\dots, p\} \;  \mid  \;   n_{i,r} > 0\} }  \frac{ n_{i,r}^{n_{i,r}}}{n_{i,r}!}\\
&\le \sum\limits_{j=1}^{N(k)}  \exp\left(- k\right) \prod\limits_{\{(i ,r) \in \{1,\dots,d\} \times \{1,\dots, p\} \;  \mid  \;   n_{i,r} > 0\} }    n_{i,r}^{n_{i,r}}  n_{i,r}^{-1/2}\left( \frac{e}{n_{i,r}}\right)^{n_{i,r}}\\
 &= \sum\limits_{j=1}^{N(k)} \exp(-k)  \prod\limits_{\{(i ,r) \in \{1,\dots,d\} \times \{1,\dots, p\} \;  \mid  \;   n_{i,r} > 0\} }    n_{i,r}^{-1/2} \exp({n_{i,r}})\\ 
& = \sum\limits_{j=1}^{N(k)}   \prod\limits_{\{(i ,r) \in \{1,\dots,d\} \times \{1,\dots, p\} \;  \mid  \;   n_{i,r} > 0\} }    n_{i,r}^{-1/2}. 
\end{align*}
The result follows from the previous display and Lemma \ref{pro:bound_multinomial}.  
\end{proof}

%% file: arXiv.bbl
\begin{thebibliography}{45}
\providecommand{\natexlab}[1]{#1}
\providecommand{\url}[1]{\texttt{#1}}
\expandafter\ifx\csname urlstyle\endcsname\relax
  \providecommand{\doi}[1]{doi: #1}\else
  \providecommand{\doi}{doi: \begingroup \urlstyle{rm}\Url}\fi

\bibitem[Adamczak and Bednorz(2015)]{adamczak2015exponential}
Adamczak, R. and Bednorz, W.
\newblock Exponential concentration inequalities for additive functionals of {M}arkov chains.
\newblock \emph{ESAIM: Probability and Statistics}, 19:\penalty0 440--481, 2015.

\bibitem[Adamczak and Wolff(2015)]{adamczak2015concentration}
Adamczak, R. and Wolff, P.
\newblock Concentration inequalities for non-{L}ipschitz functions with bounded derivatives of higher order.
\newblock \emph{Probability Theory and Related Fields}, 162\penalty0 (3):\penalty0 531--586, 2015.

\bibitem[Alquier et~al.(2019)Alquier, Doukhan, and Fan]{alquier2019exponential}
Alquier, P., Doukhan, P., and Fan, X.
\newblock Exponential inequalities for nonstationary {M}arkov chains.
\newblock \emph{Dependence Modeling}, 7\penalty0 (1):\penalty0 150--168, 2019.

\bibitem[Alvarez et~al.(2012)Alvarez, Rosasco, and Lawrence]{alvarez2012kernels}
Alvarez, M.~A., Rosasco, L., and Lawrence, N.~D.
\newblock Kernels for vector-valued functions: A review.
\newblock \emph{Foundations and Trends{\textregistered} in Machine Learning}, 4\penalty0 (3):\penalty0 195--266, 2012.

\bibitem[Arcones(1995)]{arcones1995bernstein}
Arcones, M.~A.
\newblock A {B}ernstein-type inequality for {U}-statistics and {U}-processes.
\newblock \emph{Statistics \& Probability Letters}, 22\penalty0 (3):\penalty0 239--247, 1995.

\bibitem[Arcones and Gine(1993)]{arcones1993limit}
Arcones, M.~A. and Gine, E.
\newblock Limit theorems for {U}-processes.
\newblock \emph{The Annals of Probability}, 21\penalty0 (3):\penalty0 1494--1542, 1993.

\bibitem[Ballarin(2024)]{ballarin2024ridge}
Ballarin, G.
\newblock Ridge regularized estimation of {VAR} models for inference.
\newblock \emph{Journal of Time Series Analysis}, \penalty0 (Special Issue):\penalty0 1--23, 2024.

\bibitem[Basu and Michailidis(2015)]{basu2015regularized}
Basu, S. and Michailidis, G.
\newblock Regularized estimation in sparse high-dimensional time series models.
\newblock \emph{The Annals of Statistics}, 43\penalty0 (4):\penalty0 1535--1567, 2015.

\bibitem[Benrhmach et~al.(2020)Benrhmach, Namir, Namir, and Bouyaghroumni]{benrhmach2020nonlinear}
Benrhmach, G., Namir, K., Namir, A., and Bouyaghroumni, J.
\newblock Nonlinear autoregressive neural network and extended kalman filters for prediction of financial time series.
\newblock \emph{Journal of Applied Mathematics}, 2020\penalty0 (1):\penalty0 5057801, 2020.

\bibitem[Borisov and Volodko(2015)]{borisov2015note}
Borisov, I. and Volodko, N.
\newblock A note on exponential inequalities for the distribution tails of canonical von {M}ises’ statistics of dependent observations.
\newblock \emph{Statistics \& Probability Letters}, 96\penalty0 (C):\penalty0 287--291, 2015.

\bibitem[Bosq(2000)]{bosq2000linear}
Bosq, D.
\newblock \emph{Linear processes in function spaces: theory and applications}, volume 149.
\newblock Springer Science \& Business Media, 2000.

\bibitem[Carrington et~al.(2014)Carrington, Fieguth, and Chen]{carrington2014new}
Carrington, A.~M., Fieguth, P.~W., and Chen, H.~H.
\newblock A new {M}ercer sigmoid kernel for clinical data classification.
\newblock In \emph{2014 36th Annual International Conference of the IEEE Engineering in Medicine and Biology Society}, pages 6397--6401. IEEE, 2014.

\bibitem[Chakrabortty and Kuchibhotla(2018)]{chakrabortty2018tail}
Chakrabortty, A. and Kuchibhotla, A.~K.
\newblock Tail bounds for canonical {U}-statistics and {U}-processes with unbounded kernels.
\newblock Technical report, Working paper, Wharton School, University of Pennsylvania, 2018.

\bibitem[Chen and Wu(2017)]{chen2017concentration}
Chen, L. and Wu, W.~B.
\newblock Concentration inequalities for empirical processes of linear time series.
\newblock \emph{Journal of Machine Learning Research}, 18:\penalty0 231--1, 2017.

\bibitem[Cline and Pu(1999)]{cline1999geometric}
Cline, D.~B. and Pu, H.-M.~H.
\newblock Geometric ergodicity of nonlinear time series.
\newblock \emph{Statistica Sinica}, 9:\penalty0 1103--1118, 1999.

\bibitem[Dedecker and Gou{\"e}zel(2015)]{dedecker2015subGaussian}
Dedecker, J. and Gou{\"e}zel, S.
\newblock {Subgaussian concentration inequalities for geometrically ergodic {M}arkov chains}.
\newblock \emph{Electronic Communications in Probability}, 20\penalty0 (64):\penalty0 1--12, 2015.

\bibitem[Duchemin et~al.(2023)Duchemin, De~Castro, and Lacour]{duchemin2023concentration}
Duchemin, Q., De~Castro, Y., and Lacour, C.
\newblock Concentration inequality for {U}-statistics of order two for uniformly ergodic {M}arkov chains.
\newblock \emph{Bernoulli}, 29\penalty0 (2):\penalty0 929--956, 2023.

\bibitem[Fan et~al.(2021)Fan, Jiang, and Sun]{fan2021hoeffding}
Fan, J., Jiang, B., and Sun, Q.
\newblock Hoeffding's inequality for general {M}arkov chains and its applications to statistical learning.
\newblock \emph{Journal of Machine Learning Research}, 22\penalty0 (139):\penalty0 1--35, 2021.

\bibitem[Fang et~al.(1994)Fang, Loparo, and Feng]{fang1994inequalities}
Fang, Y., Loparo, K.~A., and Feng, X.
\newblock Inequalities for the trace of matrix product.
\newblock \emph{IEEE Transactions on Automatic Control}, 39\penalty0 (12):\penalty0 2489--2490, 1994.

\bibitem[Han(2018)]{han2018exponential}
Han, F.
\newblock An exponential inequality for {U}-statistics under mixing conditions.
\newblock \emph{Journal of Theoretical Probability}, 31\penalty0 (1):\penalty0 556--578, 2018.

\bibitem[Hastie et~al.(2009)Hastie, Tibshirani, Friedman, and Friedman]{hastie2009elements}
Hastie, T., Tibshirani, R., Friedman, J.~H., and Friedman, J.~H.
\newblock \emph{The elements of statistical learning: data mining, inference, and prediction}, volume~2.
\newblock Springer, 2009.

\bibitem[Joshi and Bissu(1996)]{JOSHI1996251}
Joshi, C. and Bissu, S.
\newblock Inequalities for some special functions.
\newblock \emph{Journal of Computational and Applied Mathematics}, 69\penalty0 (2):\penalty0 251--259, 1996.

\bibitem[Karvonen et~al.(2023)Karvonen, Cockayne, Tronarp, and S{\"a}rkk{\"a}]{karvonen2023probabilistic}
Karvonen, T., Cockayne, J., Tronarp, F., and S{\"a}rkk{\"a}, S.
\newblock A probabilistic {T}aylor expansion with {G}aussian processes.
\newblock \emph{Transactions on Machine Learning Research}, 2023\penalty0 (8), 2023.

\bibitem[Kato et~al.(2006)Kato, Taniguchi, and Honda]{kato2006statistical}
Kato, H., Taniguchi, M., and Honda, M.
\newblock Statistical analysis for multiplicatively modulated nonlinear autoregressive model and its applications to electrophysiological signal analysis in humans.
\newblock \emph{IEEE Transactions on Signal Processing}, 54\penalty0 (9):\penalty0 3414--3425, 2006.

\bibitem[Liu and Li(2023)]{liu2023estimation}
Liu, Z. and Li, M.
\newblock On the estimation of derivatives using plug-in kernel ridge regression estimators.
\newblock \emph{Journal of Machine Learning Research}, 24\penalty0 (266):\penalty0 1--37, 2023.

\bibitem[Loh and Wainwright(2012)]{loh2012high}
Loh, P.-L. and Wainwright, M.~J.
\newblock High-dimensional regression with noisy and missing data: Provable guarantees with nonconvexity.
\newblock \emph{The Annals of Statistics}, 40\penalty0 (3):\penalty0 1637--1664, 2012.

\bibitem[Lu and Jiang(2001)]{lu2001l1}
Lu, Z. and Jiang, Z.
\newblock L1 geometric ergodicity of a multivariate nonlinear {AR} model with an {ARCH} term.
\newblock \emph{Statistics \& Probability Letters}, 51\penalty0 (2):\penalty0 121--130, 2001.

\bibitem[Luke(1972)]{LUKE197241}
Luke, Y.~L.
\newblock Inequalities for generalized hypergeometric functions.
\newblock \emph{Journal of Approximation Theory}, 5\penalty0 (1):\penalty0 41--65, 1972.

\bibitem[Meyn and Tweedie(2012)]{meyn2012markov}
Meyn, S.~P. and Tweedie, R.~L.
\newblock \emph{Markov chains and stochastic stability}.
\newblock Springer Science \& Business Media, 2012.

\bibitem[Nyberg(2018)]{nyberg2018forecasting}
Nyberg, H.
\newblock Forecasting {US} interest rates and business cycle with a nonlinear regime switching {VAR} model.
\newblock \emph{Journal of Forecasting}, 37\penalty0 (1):\penalty0 1--15, 2018.

\bibitem[Paulin(2015)]{paulin2015concentration}
Paulin, D.
\newblock Concentration inequalities for {M}arkov chains by {M}arton couplings and spectral methods.
\newblock \emph{Electronic Journal of Probability}, 20\penalty0 (79):\penalty0 1--32, 2015.

\bibitem[Pitcan(2017)]{pitcan2017note}
Pitcan, Y.
\newblock A note on concentration inequalities for {U}-statistics.
\newblock \emph{arXiv preprint arXiv:1712.06160}, 2017.

\bibitem[Rudelson and Vershynin(2013)]{rudelson2013hanson}
Rudelson, M. and Vershynin, R.
\newblock Hanson-{W}right inequality and sub-{G}aussian concentration.
\newblock \emph{Electronic Communications in Probability}, 18:\penalty0 1--9, 2013.

\bibitem[Sch{\"o}lkopf et~al.(2001)Sch{\"o}lkopf, Herbrich, and Smola]{scholkopf2001generalized}
Sch{\"o}lkopf, B., Herbrich, R., and Smola, A.~J.
\newblock A generalized representer theorem.
\newblock In \emph{International conference on computational learning theory}, pages 416--426. Springer, 2001.

\bibitem[Sharma et~al.(2017)Sharma, Sharma, and Athaiya]{sharma2017activation}
Sharma, S., Sharma, S., and Athaiya, A.
\newblock Activation functions in neural networks.
\newblock \emph{Towards Data Sci}, 6\penalty0 (12):\penalty0 310--316, 2017.

\bibitem[Shen et~al.(2020)Shen, Han, and Witten]{shen2020exponential}
Shen, Y., Han, F., and Witten, D.
\newblock Exponential inequalities for dependent {V}-statistics via random {F}ourier features.
\newblock \emph{Electronic Journal Probability}, 25\penalty0 (7):\penalty0 1--18, 2020.

\bibitem[Smale and Zhou(2005)]{smale2005shannon}
Smale, S. and Zhou, D.-X.
\newblock Shannon sampling {II}: Connections to learning theory.
\newblock \emph{Applied and Computational Harmonic Analysis}, 19\penalty0 (3):\penalty0 285--302, 2005.

\bibitem[Steinwart and Christmann(2008)]{steinwart2008support}
Steinwart, I. and Christmann, A.
\newblock \emph{Support Vector Machines}.
\newblock Information Science and Statistics. Springer New York, 2008.

\bibitem[Vershynin(2018)]{vershynin2018high}
Vershynin, R.
\newblock \emph{High-dimensional probability: An introduction with applications in data science}, volume~47.
\newblock Cambridge university press, 2018.

\bibitem[Wahba(1990)]{wahba1990spline}
Wahba, G.
\newblock Spline models for observational data.
\newblock In \emph{CBMS-NSF Regional Conference Series in Applied Mathematics}, volume~59, 1990.

\bibitem[Wainwright(2019)]{wainwright2019high}
Wainwright, M.~J.
\newblock \emph{High-dimensional statistics: A non-asymptotic viewpoint}, volume~48.
\newblock Cambridge University Press, 2019.

\bibitem[Winkelbauer(2012)]{winkelbauer2012moments}
Winkelbauer, A.
\newblock Moments and absolute moments of the normal distribution.
\newblock \emph{arXiv preprint arXiv:1209.4340}, 2012.

\bibitem[Yu et~al.(2021)Yu, Liu, Heck, Berger, and Song]{yu2021sparse}
Yu, P.-N., Liu, C.~Y., Heck, C.~N., Berger, T.~W., and Song, D.
\newblock A sparse multiscale nonlinear autoregressive model for seizure prediction.
\newblock \emph{Journal of Neural Engineering}, 18\penalty0 (2):\penalty0 026012, 2021.

\bibitem[Yuan and Zhou(2016)]{yuan2016minimax}
Yuan, M. and Zhou, D.-X.
\newblock Minimax optimal rates of estimation in high dimensional additive models.
\newblock \emph{The Annals of Statistics}, 44\penalty0 (6):\penalty0 2564--2593, 2016.

\bibitem[Zhou and Raskutti(2018)]{zhou2018non}
Zhou, H.~H. and Raskutti, G.
\newblock Non-parametric sparse additive auto-regressive network models.
\newblock \emph{IEEE Transactions on Information Theory}, 65\penalty0 (3):\penalty0 1473--1492, 2018.

\end{thebibliography}
